\documentclass{amsart}
\usepackage{latexsym, amssymb, stmaryrd, eucal}
\usepackage{amsopn}
\usepackage[T1]{fontenc}

\DeclareTextSymbol{\thh}{T1}{254}
\def\th{\textnormal{\thh}}

\newtheorem{thm}{Theorem}[subsection]
\newtheorem{lemma}[thm]{Lemma}
\newtheorem{prop}[thm]{Proposition}
\newtheorem{cor}[thm]{Corollary}

\newtheorem{result}[thm]{Result}

\newtheorem{df}[thm]{Definition}
\newtheorem{rmk}[thm]{Remark}
\newtheorem{rmks}[thm]{Remarks}

\newtheorem{question}[thm]{Question}

\newenvironment{renumerate}
        {
         \begin{enumerate}}

\newcommand{\BB}{\mathbb{B}}  
\newcommand{\BK}{\mathbb{K}}  
\newcommand{\BN}{\mathbb{N}}

\newcommand{\cu}[1]{\mathcal{#1}}

\newcommand{\bo}[1]{\mathbf{#1}}

\newcommand{\ti}[1]{\widetilde{#1}}

\newcommand{\sa}[1]{\mathsf{#1}}

\renewcommand{\hat}{\widehat}

\makeatletter

\def\indsym#1#2{%
  \setbox0=\hbox{$\m@th#1x$}%
  \kern\wd0%
  \hbox to 0pt{\hss$\m@th#1\mid$\hbox to 0pt{$\m@th#1^{#2}$}\hss}%
  \lower.9\ht0\hbox to 0pt{\hss$\m@th#1\smile$\hss}%
  \kern\wd0}
\newcommand{\ind}[1][]{\mathop{\mathpalette\indsym{#1}}}

\def\nindsym#1#2{%
  \setbox0=\hbox{$\m@th#1x$}%
  \kern\wd0%
  \hbox to 0pt{\hss$\m@th#1\not$\kern1.4\wd0\hss}
  \hbox to 0pt{\hss$\m@th#1\mid$\hbox to 0pt{$\m@th#1^{\,#2}$}\hss}%
  \lower.9\ht0\hbox to 0pt{\hss$\m@th#1\smile$\hss}%
  \kern\wd0}
\newcommand{\nind}[1][]{\mathop{\mathpalette\nindsym{#1}}}

\def\dotminussym#1#2{%
  \setbox0=\hbox{$\m@th#1-$}%
  \kern.5\wd0%
  \hbox to 0pt{\hss\hbox{$\m@th#1-$}\hss}%
  \raise.6\ht0\hbox to 0pt{\hss$\m@th#1.$\hss}%
  \kern.5\wd0}
\newcommand{\dotminus}{\mathbin{\mathpalette\dotminussym{}}}

\def \<{\langle}
\def \>{\rangle}
\def \((  {(\!(}
\def \)) {)\!)}

\def \cl {\operatorname{cl}}

\def \tp{\operatorname{tp}}

\def \acl{\operatorname{acl}}
\def \dcl{\operatorname{dcl}}

\def \APr{\operatorname{APr}}
\def \div{\operatorname{div}}

\def \as {\operatorname{a.s}}
\def \fo{\operatorname{fdcl}}

\def \l{\llbracket}
\def \rr{\rrbracket}

\def\thind{\ind[\th]}

\begin{document}

\title{Independence in Randomizations}

\author{Uri Andrews}

\address{Department of Mathematics\\
University of Wisconsin, 480 Lincoln Drive, Madison, WI 53706, USA}
\email{andrews@math.wisc.edu}

\author{Isaac Goldbring}

\address{Department of Mathematics, Statistics, and Computer Science\\
University of Illinois at Chicago, Science and Engineering Offices (M/C 249), \\
851 S. Morgan St., Chicago, IL 60607-7045, USA and Department of Mathematics\\University of California, Irvine, 340 Rowland Hall (Bldg.\# 400),
Irvine, CA 92697-3875}
\email{isaac@math.uci.edu}

\author{H. Jerome Keisler}

\address{Department of Mathematics\\
University of Wisconsin, 480 Lincoln Drive, Madison WI 53706, USA}
\email{keisler@math.wisc.edu}
\date{\today}

\begin{abstract}
The randomization of a complete first order theory $T$ is the complete continuous theory $T^R$
with two sorts, a sort for random elements of models of $T$, and a sort for events in an underlying atomless probability space.
We study independence relations and related ternary relations on the randomization of $T$.
We show that if $T$ has the exchange property and $\acl=\dcl$, then $T^R$ has a strict independence relation in the home sort, and hence is real rosy. In particular, if $T$ is o-minimal, then $T^R$ is real rosy.
\end{abstract}


\maketitle

\section{Introduction}

A randomization of a first order structure $\cu M$, as introduced by Keisler [Ke1] and formalized as a metric structure by Ben Yaacov and Keisler [BK], is a continuous structure $\cu N$ with two sorts, a sort for random elements of $\cu M$, and a sort for events in an underlying atomless probability space.
For each $n$-tuple $\bar{ a}$ of random elements of $\cu M$ and each first order
formula $\varphi(\bar v)$, the set of points in the underlying probability space where $\varphi(\bar{ a})$ is true is an event in $\cu N$ denoted by $\l\varphi(\bar{ a})\rr$.
Given a complete first order theory $T$, the theory $T^R$ of randomizations of models of $T$ forms a complete theory in continuous logic, which is called the randomization of $T$.

In general, the theories $T$ and $T^R$ share many model-theoretic features.  It was shown in [BK] that $T$ is stable if and only if $T^R$ is stable.  Unfortunately, the analogous result for simplicity in place of stability is false, as was shown in [Be2].  More precisely,
$T^R$ is simple only if $T^R$ is stable.

It is apparent that a model theoretic notion such as o-minimality of $T$ is not preserved in $T^R$, as $T^R$ will not even be an ordered theory.  We show in this paper that
if $T$ is o-minimal, then $T^R$ retains some model theoretic tameness. In particular, we show that if $T$ is o-minimal, then $T^R$  has a strict independence relation.
We will actually show somewhat more.  It is well known that every o-minimal theory has $\acl=\dcl$ and the  exchange property (for algebraic closure).  Our main theorem shows that if
$T$ has $\acl=\dcl$ and the  exchange property, then $T^R$ has a strict independence relation that has local character with the smallest possible bound.

In view of the results on stable and simple theories stated above, it is natural to ask: \emph{Is $T$ rosy if and only if $T^R$ is rosy?}
Recall that a first order theory $T$ is \emph{rosy} if $T^{\operatorname{eq}}$ possesses a strict  independence relation.  $T$ is \emph{real rosy} if $T$ has a strict independence relation (so rosy implies real rosy).
Moreover, every real rosy theory has a weakest strict independence relation,  namely \emph{thorn independence}, a notion first introduced by Thomas Scanlon.   Classical rosy theories were first studied in the theses of Alf Onshuus and Clifton Ealy and are a common generalization of simple theories and o-minimal theories.  Continuous rosy theories were first studied in the paper of Ealy and Goldbring in [EG].
It is shown in [EG] that for every real rosy continuous theory, thorn independence is the weakest strict countable independence relation.  Our main theorem in this paper shows that if $T$ is o-minimal, then $T^R$ is real rosy.

On the way to our main result,  we study various notions of independence in models of $T^R$, including algebraic independence, dividing independence, and thorn independence.
We introduce the ``pointwise'' version of an arbitrary notion of independence on models of $T$, that is, the notion of independence on models of $T^R$ obtained by asking for independence almost everywhere in the underlying probability space.  By gluing together the pointwise version of thorn independence, with dividing independence restricted to the event sort, we are able to produce a strict independence relation in $T^R$ whenever $T$ has $\acl=\dcl$ and the exchange property.

We conclude this introduction with an outline of the rest of the paper.  In Section 2, we recall the relevant background from continuous logic as well as the general theory of abstract independence relations as exposited in [Ad2].  In Section 3 we show that first order theories with the exchange property are real rosy, and admit a characterization of thorn independence that facilitates the study of pointwise thorn independence.
In Section 4, we introduce the notion of a countably based independence relation and the countable union property.  For every ternary relation $\ind[I]$ with monotonicity, there is a unique countably based relation that agrees with $\ind[I]$ on countable sets.   This will aid us in defining, for a given ternary relation on small subsets of the big model of $T$, a corresponding pointwise notion.  In Section 5, we recall some basic facts about randomizations as well as some results we will need from [AGK] concerning definable and algebraic closure in models of $T^R$.  We also prove a ``downward'' result:  If $T$ has $\acl=\dcl$ and $T^R$ is real rosy, then $T$ is real rosy.

In Section 6, we begin the study of notions of independence in models of $T^R$ in earnest but restrict our attention to the event sort.
Section 7 is concerned with notions of pointwise independence.  Given a ternary relation $\ind[I]$ with monotonicity on models of $T$, $\ind[I\omega] \ \ $ is the countably based relation on small subsets of the big model of $T^R$ such that for all countable  $A,B,C$, $A\ind[I\omega]_C \ B$ holds if and only if $A(\omega)\ind[I]_{C(\omega)} B(\omega)$ holds for almost all $\omega$ in the underlying probability space.  The results of Section 4 guarantee the unique existence of $\ind[I\omega] \ \ $.  We then prove that whenever
$T$ has the exchange property, pointwise thorn independence is an independence relation but is not strict.  Lastly, in Section 8, we put together the previous results to produce a strict independence relation when $T$ has the exchange property and $\acl=\dcl$.

Continuous model theory in its current form is developed in the papers [BBHU] and [BU].  Randomizations of models are treated in [AGK], [AK], [Be2], [BK], [EG], [GL1], and [Ke1].

\section{Preliminaries on Continuous Logic}

We will follow the notation and terminology of [BK] and [AGK].
We assume familiarity with the basic notions about continuous model theory as developed
in [BBHU], including the notions of a theory, structure, pre-structure,  model of a theory,
elementary extension, isomorphism, and $\kappa$-saturated structure.
In particular, the universe of a pre-structure is a pseudo-metric space, the universe of a
structure is a complete metric space, and every pre-structure has a unique completion.
A \emph{tuple} is a finite sequence, and $A^{<\BN}$ is the set of all tuples of elements of $A$.
We use the word ``countable'' to mean of cardinality at most $\aleph_0$.
We assume throughout that $L$ is a countable first order signature, and let $[L]$ denote the set of formulas of $L$.
We will sometimes write $\varphi(A)$ for a first order formula with finitely many parameters in a set $A$,
and use similar notation for more than one parameter set.

Throughout this paper, $T$ will denote a complete first order theory, and $U$ will denote a complete first order or continuous theory.
 $\upsilon$ denotes an uncountable inaccessible cardinal that is held fixed.  By a
\emph{big model} of  $U$ we mean a saturated model $\cu N\models U$ of cardinality $|\cu N|=\upsilon$ (or finite).
Thus every complete theory has a unique big model up to isomorphism.  For this reason, we sometimes refer to ``the'' big model of  $U$.
We call a set \emph{small} if it has cardinality $< \upsilon$, and \emph{large} otherwise.
$\cu M$ will denote the big model of $T$ with universe $M$.  In Sections 2 through 4, $\cu N$ will denote the big model of $U$ (so $\cu N$ is a first order
structure if $U$ is a first order theory, and a continuous structure if $U$ is a continuous theory).

To avoid the assumption that uncountable inaccessible cardinals exist, one can instead assume only that $\upsilon=\upsilon^{\aleph_0}$
and take a big model to be an $\upsilon$-universal domain, as in [BBHU], Definition 7.13.  With that approach, a big model exists but is not unique.

\subsection{Definability}

\begin{df}   In first order logic, a formula $\varphi(u,\bar v)$ is \emph{functional} in  $T$ if
$$T\models(\forall \bar v)(\exists ^{\le  1} u)\varphi(u,\bar v).$$
$\varphi(u,\bar v)$ is \emph{algebraical} in $T$ if there exists $n\in\BN$ such that
$$T\models(\forall \bar v)(\exists ^{\le  n} u)\varphi(u,\bar v).$$
\end{df}

The \emph{definable closure} of $A$ in  $\cu M$  is the set
$$\dcl^{\cu M}(A)=\{b\in M\mid \cu M\models\varphi(b,\bar a) \mbox{ for some functional } \varphi \mbox{ and } \bar a\in A^{<\BN}\}.$$
The \emph{algebraic closure} of $A$ in $\cu M$ is the set
$$\acl^{\cu M}(A)=\{b\in M\mid \cu M\models\varphi(b,\bar a) \mbox{ for some algebraical } \varphi \mbox{ and } \bar a\in A^{<\BN}\}.$$

We refer to [BBHU] for the definitions of the algebraic closure $\acl^{\cu N}(A)$ and definable closure $\dcl^{\cu N}(A)$ in a continuous structure $\cu N$.
If  $\cu N$ is clear from the context, we will sometimes drop the superscript and write $\dcl, \acl$ instead of $ \dcl^\cu N, \acl^\cu N$.
We will often use the following facts without explicit mention.

\begin{result}  \label{f-algebraic-cardinality}  (Follows from [BBHU], Exercise 10.8)
For every set $A$, $\acl(A)$ has cardinality at most $( |A|+2)^{\aleph_0}$.  Thus the algebraic closure of a small set is small.
\end{result}

\begin{result} \label{f-definableclosure}  (Definable Closure, Exercises 10.10 and 10.11, and Corollary 10.5 in [BBHU])
\begin{enumerate}
\item If $ A\subseteq\cu N$ then $\dcl( A)=\dcl(\dcl( A))$ and $\acl( A)=\acl(\acl( A))$.
\item If $ A$ is a dense subset of the topological closure of $B$ and $ B\subseteq\cu N$, then $\dcl(A)=\dcl( B)$ and $\acl(A)=\acl( B)$.
\end{enumerate}
\end{result}

It follows that for any $ A\subseteq\cu N$, $\dcl( A)$ and $\acl( A)$ are topologically closed.

\subsection{Abstract Independence Relations}

Since the various properties of independence are given some slightly different names in
various parts of the literature, we take this opportunity to declare that we are following the terminology established in [Ad2], which is
repeated here for the reader's convenience.  In this paper, we will sometimes write $AB$ for $A\cup B$, and write $[A,B]$ for
$\{D\mid A\subseteq D \wedge D\subseteq B\}.$

\begin{df}[Adler]  Let $\cu N$ be the big model of $U$.  By a \emph{ternary relation for $U$}, or a \emph{ternary relation over $\cu N$},  we mean
a ternary relation $\ind$ on the small subsets of $\cu N$.  We say that $\ind$
is an \emph{independence relation} if it satisfies the following \emph{axioms for independence relations} for all small sets:
\begin{renumerate}
\item (Invariance) If $A\ind_CB$ and $(A',B',C')\equiv (A,B,C)$, then $A'\ind_{C'}B'$.
\item (Monotonicity) If $A\ind_C B$, $A'\subseteq A$, and $B'\subseteq B$, then $A'\ind_C B'$.
\item (Base monotonicity) Suppose $C\in[D, B]$.  If $A\ind_D B$, then $A\ind_C B$.
\item (Transitivity) Suppose $C\in[D, B]$.  If $B\ind_C A$ and $C\ind_D A$, then $B\ind_D A$.
\item (Normality) $A\ind_CB$ implies $AC\ind_C B$.
\item (Extension) If $A\ind_C B$ and $\hat B\supseteq B$, then there is $A'\equiv_{BC} A$ such that $A'\ind_C \hat B$.
\item (Finite character) If $A_0\ind_CB$ for all finite $A_0\subseteq A$, then $A\ind_CB$.
\item (Local character) For every $A$, there is a cardinal $\kappa(A)<\upsilon$ such that, for every $B$, there is a subset $C$ of $B$ with $|C|<\kappa(A)$ such that $A\ind_CB$.
\end{enumerate}
\end{renumerate}
We will refer to the first five axioms (i)--(v) as the \emph{basic axioms}.
\end{df}

As the trivial independence relation (which declares $A\ind_CB$ to always hold) is obviously of little interest, one adds an extra
condition to avoid such trivialities.

\begin{df}
An independence relation $\ind$ is \emph{strict} if it satisfies
\begin{itemize}
\item[(ix)] (Anti-reflexivity) $a\ind_B a$ implies $a\in \acl(B)$.
\end{itemize}
\end{df}

There are four other useful properties to consider when studying ternary relations over $\cu N$:
The stationarity property is taken from [BU], Theorem 8.10.

\begin{df}
\noindent\begin{enumerate}
\item[(x)] (Countable character) If $A_0\ind_CB$ for all countable $A_0\subseteq A$, then $A\ind_CB$.
\item[(xi)] (Full existence) For every small $A,B,C$, there is $A'\equiv_C A$ such that $A'\ind_C B$.
\item[(xii)] (Symmetry) For every small $A,B,C$, $A\ind_CB$ implies $B\ind_CA$.
\item[(xiii)] (Stationarity)  For all small $B, C$ and tuples $\bar a, \bar g$, if $C$ is algebraically closed,
$\bar a\equiv_C \bar g$, $\bar a\ind_C BC$, and  $\bar g\ind_C BC$, then $\bar a\equiv_{BC} \bar g$.
\end{enumerate}
\end{df}

\begin{df}  We say that $\ind$ is a \emph{countable independence relation} if it satisfies all the axioms for an independence relation except that finite character is replaced by countable character.  A \emph{strict countable independence relation} is a countable independence relation that satisfies anti-reflexivity.
\end{df}

\begin{rmks}  \label{r-fe-vs-ext}

\noindent\begin{itemize}
\item[(1)]  Whenever $\ind$ satisfies invariance, monotonicity, transitivity, normality, full existence, and symmetry,
then $\ind$ also satisfies extension (Remark 1.2 in [Ad2]).
\item[(2)] If $\ind$ satisfies base monotonicity and local character, then $A\ind_C C$ for all small $A,C$ (Appendix to [Ad1]).
\item[(3)] If $\ind$ satisfies monotonicity and extension, and $A\ind_C C$ holds for all small $A,C$, then $\ind$ also satisfies full existence
(Appendix to [Ad1]).
\item[(4)]  Any countable independence relation has symmetry.
\end{itemize}
\end{rmks}

In the first order setting, Remarks \ref{r-fe-vs-ext} are proved in [Ad1] and [Ad2].  It is straightforward to check that these proofs persist in the continuous setting.  Theorem 2.5 in [Ad2] shows that any independence
relation has symmetry.  The same argument with Morley sequences of length $\omega_1$ instead of countable Morley sequences proves (4).

\begin{result}  \label{r-pairs} (Pairs Lemma)  Suppose $\ind$ has monotonicity, base monotonicity, transitivity, and symmetry.  Then
$$AD\ind_C B \Leftrightarrow A\ind_{CD} BD \wedge D\ind_C B.$$
\end{result}

For a proof of Result \ref{r-pairs}, see, for example, Proposition 17 in [GL2].

\begin{df}

\noindent\begin{enumerate}
\item We say that $\ind$ has \emph{countably local character} if for every countable set $A$ and every small set $B$, there is a countable subset $C$ of $B$ such that $A\ind_CB$.
\item We say that $\ind$ has \emph{small local character} if  for all small sets $A, B, C_0$
such that $C_0\subseteq B$ and $|C_0|\le|A|+\aleph_0$, there is
a set $C\in[C_0,B]$ such that $|C|\le|A|+\aleph_0$ and $A\ind_C \ B$.
\end{enumerate}
\end{df}

\begin{rmk}  \label{r-countably-localchar}

\noindent\begin{enumerate}
\item If $\ind$ has small local character, then $\ind$ has local character with bound $\kappa(D)=(|D|+\aleph_0)^+$ (the smallest possible bound).
In the presence of base monotonicity, the converse is also true.
\item If $\ind$ has local character with bound $\kappa(D)=(|D|+\aleph_0)^+$, then $\ind$ has countably local character.
\item (Compare with Remark 1.3 in [Ad2]) If $\ind$ has invariance, countable character, base monotonicity, and countably local character, then
$\ind$ has local character with bound $$\kappa(D)=((|D| + \aleph_0)^{\aleph_0})^+.$$
\end{enumerate}
\end{rmk}

\begin{proof}  (1) and (2) are obvious.

(3)  Fix a small set $D$.  By countably local character, for each countable subset $A$ of $D$ and any small set $B$, there is a countable subset $C(A,B)$ of $B$
such that $A\ind_{C(A,B)} \ B$.  Let $C=\bigcup\{C(A,B)\mid A\subseteq D, |A|\le\aleph_0\}$.  Then $|C|\le (|D| + \aleph_0)^{\aleph_0}$.
By base monotonicity, we have $A\ind_C B$ for each countable $A\subseteq D$.
By countable character, $D\ind_C B$, so $\ind$ has local character with bound $\kappa(D)$.
\end{proof}

We say that $\ind[J]$ is \emph{weaker than} $\ind[I]$, and write $\ind[I]\Rightarrow \ind[J]$, if $A\ind[I]_C B\Rightarrow A\ind[J]_C B$.

\begin{rmk}  \label{r-weaker}  Suppose $\ind[I]\Rightarrow\ind[J]$.  If $\ind[I]$ has full existence, local character,
countably local character, or small local character, then $\ind[J]$ \ has the same property.
\end{rmk}

\subsection{Special Independence Relations}

\noindent  In this paper we will introduce many special ternary relations in both first order and continuous logic.  We will be interested in which of the independence axioms hold for these relations, and sometimes call them notions of independence.  At the end of the paper we include an appendix where the reader can review these notions of independence.

 In both first order and continuous logic, we define the notion of \emph{algebraic independence}, denoted $\ind[a]$,
by setting $A\ind[a]_CB$ to mean $\acl(AC)\cap \acl(BC)=\acl(C)$. In first order logic, $\ind[a]$ satisfies all axioms for a strict
independence relation except perhaps for base monotonicity.

\begin{prop}  \label{p-alg-indep}
In continuous logic, $\ind[a]$ satisfies symmetry and all axioms for a strict countable independence relation except perhaps for
base monotonicity and extension.
\end{prop}

\begin{proof}
The proof is exactly as in [Ad2], Proposition 1.5, except for some minor modifications.  For example, countable character of $\acl$ in
continuous logic yields countable character of $\ind[a]$.  Also, in the verification of local character, one needs to take
$\kappa(A):=((|A|+2)^{\aleph_0})^+$ instead of $(|A|+\aleph_0)^+$.
\end{proof}

Recall the following definitions from [Ad2]:

\begin{df}  \label{d-special-ind}
Suppose that $A,B,C$ are small subsets of the big model of $U$.
\begin{itemize}
\item $A\ind[M]_CB$ iff for every $D\in[C,\acl(BC)]$, we have $A\ind[a]_{D}B$.
\item $A\thind_CB$ iff for every (small) $E\supseteq BC$, there is $A'\equiv_{BC}A$ such that $A'\ind[M]_CE$.
\end{itemize}
\end{df}

Note that $\thind\Rightarrow\ind[M] \ $ and $\ind[M]\Rightarrow\ind[a]$.  Also, $\thind = \ind[M] \ \ $ if and only if $\ind[M] \ $ satisfies extension.
In [Ad2], it is shown that, in the first order setting, $\ind[M]$ \ satisfies all of the axioms for a strict independence relation except
perhaps for local character and extension (but now base monotonicity has been ensured).  It is shown in [EG] that
this fact remains true in continuous logic, except for finite character being replaced by countable character.

By [Ad1] and [Ad2], in the first order setting
$\thind$ satisfies all of the axioms for a strict independence relation except perhaps  local
character\footnote{Finite character follows from Proposition A.2 of [Ad1], and is stated explicitly in [Ad3], Proposition 1.3.}
(but now extension has been ensured).  It is shown in [EG] that in the continuous setting,
$\thind$ satisfies all of the axioms for a strict countable independence relation except perhaps for countable character and local character.

The theory $U$ is said to be \emph{real rosy} if $\thind$ in $U$ has local character.
Examples of real rosy theories are the first order or continuous stable theories ([BBHU, Section 14), the first order or continuous simple theories (see [Be3]),
and the first order o-minimal theories (see [On]).

The following results are consequences of Theorem 3.2 and the preceding discussion in [EG].  The proof of part (2) is the same as the proof of Remark
4.1 in [Ad2].

\begin{result}  \label{f-weakest}

\noindent\begin{enumerate}
\item $U$ is real rosy if and only if $\thind$ is a strict countable independence relation.
\item
If $\ind[I]$ satisfies the basic axioms and extension, symmetry, and anti-reflexivity, then $\ind[I]\Rightarrow\thind$.
\item If a theory has a strict countable independence relation, then it is real rosy, and $\thind$ is the weakest strict countable
independence relation.
\end{enumerate}
\end{result}

A continuous formula $\Phi(\bar x,B,C)$  \emph{divides over} $C$ if, in the big model of $U$,
there is a $C$-indiscernible sequence $\<B^i\>_{i\in\BN}$
such that $B^0\equiv_C B$ and the set of statements $\{\Phi(\bar x,B^i,C)=0\mid i\in\BN\}$ is not satisfiable.
The \emph{dividing independence relation} $A\ind[d]_C B$ is defined to hold if there is no tuple $\bar a\in A^{<\BN}$ and continuous formula
$\Phi(\bar x,B,C)$ such that $\Phi(\bar a,B,C)=0$ and $\Phi(\bar x, B,C)$ divides over $C$.

The next result is an easy consequence of Theorem 8.10 in [BU] and Theorems 14.12 and 14.14 in [BBHU].

\begin{result}  \label{f-stable-indep}
If $U$ is stable, then on the big model of $U$, $\ind[d]$ is the unique strict independence relation  that has small local character
and stationarity.
\end{result}

\begin{cor}  \label{c-simple-thorn}
If $U$ is stable, then $\ind[d]\Rightarrow\thind$ in the big model of $U$.
\end{cor}

\begin{proof}
By Results \ref{f-weakest} and \ref{f-stable-indep}.
\end{proof}

\section{First Order Theories with the Exchange Property}  \label{s-exchange}

\subsection{Blanketing and Exchange Independence}

We  introduce two ternary relations on $\cu M$  that will be useful technical tools for proving that other relations have small local
character and full existence.

\begin{df}  We define the relation $A\ind[b]_C \ B$ on $\cu M$, read ``$C$ blankets $A$ in $B$'',  to hold
if and only if for every first order formula
$\varphi(\bar x,\bar y,\bar z)\in[L]$ and all tuples $\bar{ a}\in A^{|\bar x|}$, $\bar{ b}\in B^{|\bar y|}$ and  $\bar{ c}\in C^{|\bar z|}$,
there exists $\bar{ d}\in C^{|\bar y|}$ such that
$$ \cu M\models\varphi(\bar{ a},\bar{ b},\bar{ c})\Rightarrow\varphi(\bar{ a},\bar{ d},\bar{ c}).$$
\end{df}

\begin{lemma}  \label{l-ind[b]-easy}
The relation $\ind[b]$ has  invariance, monotonicity, normality, finite character, and small local character, and $\ind[b]\Rightarrow\ind[a]$.
\end{lemma}

\begin{proof}  We prove $\ind[b]\Rightarrow\ind[a]$.  The other parts are straightforward and left to the reader.

 Suppose $A, B, C$ are small  and $A\ind[b]_C B$ in $\cu M$.  Let $e\in\acl^{\cu M}(AC)\cap\acl^{\cu M}(BC)$.  Then there are algebraical formulas
$\varphi(u,\bar x,\bar z), \psi(u,\bar y,\bar w)$ and tuples $\bar a\in A^{<\BN}, \bar b\in B^{<\BN}, \bar c,\bar {g}\in C^{<\BN}$ such that
$$\cu M\models\varphi(e,\bar a,\bar c)\wedge\psi(e,\bar b,\bar {g})$$
and
$$(\forall u\in M)[\cu M\models \varphi(u,\bar a,\bar c)\Rightarrow \tp(u/AC)=\tp(e/AC)].$$
Then
$$\cu M\models(\exists u)[\varphi(u,\bar a,\bar c)\wedge\psi(u,\bar b,\bar {g})].$$
Since $A\ind[b]_C B$, there exists $\bar d\in C^{<\BN}$ such that
$$\cu M\models(\exists u)[\varphi(u,\bar a,\bar c)\wedge\psi(u,\bar d,\bar {g})].$$
Therefore
$$\cu M\models\psi(e,\bar d,\bar {g}),$$
so $e\in\acl^{\cu M}(C).$
\end{proof}

Note that transitivity is not mentioned in Lemma \ref{l-ind[b]-easy}.

Following Grossberg and Lessman [GL2],
we define a ternary relation $A\ind[e]_C B$ for $T$, that we will call ``exchange independence''.
([GL2] gave essentially the same definition for arbitrary pre-geometries).  We will see below that  $T$ has the exchange property if and only if $\ind[e]$ is equivalent to $\thind$.
This will be useful because $\ind[e]$ is easier to handle than $\thind$.

\begin{df}  Let $A, B, C$ be small subsets of $M$.  We define $A\ind[e]_C B$ to mean that for all finite $\bar a\in A^{<\BN}$,
$$ A\cap \acl(\bar a B C)\subseteq \acl(\bar a C).$$
\end{df}

We stress that in the previous definition $\bar a$ may be empty.

\begin{lemma}  \label{l-ind[e]-char}
$A\ind[e]_C B$ if and only if
\begin{equation}  \label{eq-ind[e]}
(\forall D\in[C,AC] )(A\cap\acl(BD)\subseteq\acl(D)).
\end{equation}
\end{lemma}

\begin{proof}  Assume (1) and let $\bar a\in A^{<\BN}$ and $e\in A\cap\acl(\bar a BC)$.  Then $D=\bar a C\in[C,AC]$ and $e\in \acl(BD)$, so $e\in\acl(\bar a C)$  by (1).

Now assume $A\ind[e]_C B$ and let $D\in[C,AC]$ and $e\in A\cap\acl(DB)$.  Then for some $\bar a\in (A\cap D)^{<\BN}$, $e\in\acl(\bar a CB)$, and because $A\ind[e]_C B$
we have $e\in\acl(\bar a C)\subseteq\acl(D)$.
\end{proof}

The condition (1) in the preceding lemma was also considered in [Ad3], where it was denoted  by $B\ind[\Psi_m]_C \ A$.  If follows from Proposition 2.1 of [Ad3]
that $\ind[e]$ satisfies the basic axioms for independence and (strong) finite character.

\begin{cor}  \label{c-ind[M]-implies-ind[e]}
\noindent\begin{itemize}
\item[(1)] If $B\ind[M]_C \ A$ then $A\ind[e]_C B$.
\item[(2)]  If $T$ is real rosy, then $\thind\Rightarrow\ind[e]$.
\end{itemize}
\end{cor}

\begin{proof}
(1):  Assume $B\ind[M]_C \ A$ and let $D\in[C,AC].$  By the definition of $\ind[M] \ $, we have $B\ind[a]_D A.$  Then $A\ind[a]_D B,$  and hence
$$A\cap\acl(BD)\subseteq\acl(AD)\cap\acl(BD)\subseteq\acl(D).$$
Therefore by Lemma \ref{l-ind[e]-char} we have $A\ind[e]_C B.$

(2): This follows from (1), and the facts that $\thind\Rightarrow \ind[M] \ $, and if $T$ is real rosy then $\thind$ has symmetry.
\end{proof}

\begin{lemma}  \label{l-ind[b]-implies-ind[e]}
$\ind[b]\Rightarrow\ind[e].$
\end{lemma}

\begin{proof}  Suppose $A\ind[b]_C B$.  Let  $\bar a\in A^{<\BN}$ and $e\in A\cap\acl(\bar a BC)$.
Then there is an algebraical formula $\varphi(u,\bar x,\bar y,\bar z)$ and $\bar b\in B^{|\bar y|}, \bar c\in C^{|\bar z|}$
such that $\varphi(e,\bar a,\bar b,\bar c)$.  Since $A\ind[b]_C B$, there exists $\bar d\in C^{|\bar y|}$ such that
$\varphi(e,\bar a,\bar d,\bar c)$.  Therefore $e\in\acl(\bar a C)$.
\end{proof}

\begin{lemma}  \label{l-e-strict}
$\ind[e]$ satisfies anti-reflexivity.
\end{lemma}

\begin{proof}  Taking $\bar a$ to be empty in the definition of $\ind[e]$, we see that $A\ind[e]_C A $ if and only if $A\subseteq \acl(C).$
\end{proof}

\subsection{The Exchange Property}

\begin{df}  A first order theory $T$ has the \emph{exchange property} if
$$a\in\acl(bC)\setminus\acl(C)\Rightarrow b\in\acl(aC),$$
or equivalently, $(M,\acl^{\cu M})$ is a pregeometry.
\end{df}

 Note that, in particular, every o-minimal theory has the exchange property.  In 1994, Hrushovski and Pillay ([HP], Remarks 2.2)
stated that the exchange property ``gives rise to a notion of independence in geometric structures''.
 Theorem \ref{t-indep-FO}
 below justifies that statement with respect to the independence axioms from [Ad2] that we are using here.

\begin{lemma}  \label{l-exchange-symmetry}
$T$ has the exchange property if and only if $\ind[e]$ for $T$ has symmetry.
\end{lemma}

\begin{proof}  By Proposition 15 in [GL2], if $T$ has the exchange property then $\ind[e]$ has symmetry.  Suppose that $\ind[e]$
has symmetry.  Let $a\in\acl(bC)\setminus\acl(C)$.  Then $a\nind[e]_C b$.  By symmetry, $b\nind[e]_C a$.  Therefore $b\in \acl(aC)\setminus\acl(C)$, so $T$
has the exchange property.
\end{proof}

\begin{thm}  \label{t-indep-FO}
\noindent\begin{itemize}
\item[(1)]  $T$ has the exchange property if and only if $T$ is real rosy and $\thind=\ind[e].$
\item[(2)]  If $T$ has the exchange property, then $\thind$ has small local character.
\item[(3)]  If $T$ is stable and $\ind[d]=\ind[e]$, then $T$ has the exchange property.
\end{itemize}
\end{thm}

\begin{proof} (1)  If $T$ is real rosy and $\thind=\ind[e]$, then $\ind[e]$ has symmetry, so $T$ has the exchange property by Lemma \ref{l-exchange-symmetry}.

Suppose $T$ has the exchange property.
By Corollary 4.3 of [Ad3] and Lemmas \ref{l-ind[e]-char} and \ref{l-e-strict}, $\ind[e]$ is a strict independence relation.
By Remark 4.1 in [Ad2], $\thind$ is the weakest strict independence relation for $T$, so $T$ is real rosy and $\ind[e]\Rightarrow\thind$.
We always have $\thind\Rightarrow \ind[M] \ $.
By Corollary \ref{c-ind[M]-implies-ind[e]} and Lemma \ref{l-exchange-symmetry}, $\ind[M] \ \Rightarrow \ind[e]$, so $\thind=\ind[M] =\ind[e].$

(2) By Lemmas \ref{l-ind[b]-easy} and \ref{l-ind[b]-implies-ind[e]} and (1), $\thind$ has small local character..

(3)  By Lemma \ref{l-exchange-symmetry} and the fact that when $T$ is stable, $\ind[d]$ has symmetry.
\end{proof}

We do not know whether $T$ is stable and has the exchange property implies that $\ind[d]=\ind[e]$.

\section{Countably Based and Countable Union Properties}

In this section we will introduce the notion of a countably based ternary relation.  The reason this notion is useful is because for
each ternary relation with monotonicity there is a unique countably based ternary relation that agrees with it on countable sets (Lemma \ref{l-cbased} (2)).
This will be important in Section \ref{s-pointwise}, where it allows us to introduce, for each ternary relation with monotonicity over the big model of  $T$,
a corresponding ``pointwise independence relation'' over the big model of the randomization $T^R$ (Definition \ref{d-measurable}).
We will also introduce another property, the countable union property.  The countably based property and the countable union property will
sometimes be useful in showing that ternary relations have finite character.

The notions and results in this section hold for both first order and continuous logic.  We will give the proofs only for continuous
logic; the proofs for first order logic are similar but simpler.
We will use $(\forall^c D)$ to mean ``for all countable $D$'', and similarly for $(\exists^c D)$.

\subsection{Countably Based Relations}

\begin{df}
We say that a ternary relation $\ind[I]$ over $\cu N$ is \emph{countably based} if for all $A, B, C$ we have
$$ A\ind[I]_C B \Leftrightarrow (\forall^c A'\subseteq A)(\forall^c B'\subseteq B)(\forall^c C'\subseteq C)(\exists^c D\in[C',C])A'\ind[I]_{D} B'.$$
\end{df}

Note that if $\ind[I]$ and $\ind[J]$ are countably based and agree on countable sets, then they are the same.

\begin{df}  We say that $\ind[I]$ has \emph{two-sided countable character} if $\ind[I]$ has countable character and
$$ [(\forall^c B_0\subseteq B)A\ind[I]_C B_0] \Rightarrow A\ind[I]_C B.$$
In other words,
$$ [(\forall^c A_0\subseteq A)(\forall^c B_0\subseteq B) A_0\ind[I]_C B_0] \Rightarrow A\ind[I]_C B.$$
\end{df}

\begin{rmk} \label{r-two-sided}
If $\ind[I]$ has symmetry and countable character, then $\ind[I]$ has two-sided countable character.
\end{rmk}

\begin{proof}  Suppose $(\forall^c B_0\subseteq B)A\ind[I]_C B_0$.  By symmetry, $(\forall^c B_0\subseteq B)B_0\ind[I]_C A$.
By countable character, $B\ind[I]_C A$.  Then by symmetry again, $A\ind[I]_C B$.
\end{proof}

\begin{lemma}  \label{l-cbased}

\noindent\begin{enumerate}
\item Suppose $\ind[I]$ and $\ind[J]$ are countably based.  If
$$A\ind[I]_C B \Rightarrow A\ind[J]_C B$$
holds for all countable $A,B,C$, then it holds for all small $A,B,C$.
\item  Suppose $\ind[I]$ has monotonicity.  There is a unique ternary relation $\ind[Ic]$ \  that is countably
based and agrees with $\ind[I]$ on countable sets.  Namely,
$$ A\ind[Ic]_C\,\, B \Leftrightarrow (\forall^c A'\subseteq A)(\forall^c B'\subseteq B)(\forall^c C'\subseteq C)(\exists^c D\in[C',C])A'\ind[I]_{D} B'.$$
\item $\ind[I]$ is countably based if and only if $\ind[I]$ has monotonicity and two-sided countable character, and
whenever $A$ and $B$ are countable, we have
$$A\ind[I]_C B \Leftrightarrow (\forall^c C'\subseteq C)(\exists^c D\in[C',C]) A\ind[I]_{D} B.$$
\end{enumerate}
\end{lemma}

\begin{proof}  (1) follows easily from the definition of countably based.

(2): Uniqueness is clear.  Let $\ind[J]$\, be the relation defined by the displayed formula and let $A,B,C$ be countable.
It is obvious that $\ind[J]$ is countably based, and that $A\ind[J]_C B$ implies $A\ind[I]_C B$.  Suppose $A\ind[I]_C B$
and $A'\subseteq A, B'\subseteq B, C'\subseteq C$.  By monotonicity for $\ind[I]$ we have $A'\ind[I]_C B'$.  But $C\in[C',C]$, so $A\ind[J]_C B$
as required.

(3): Consider the following statements:
\begin{itemize}
\item[(a)] $A\ind[I]_C B$;
\item[(b)] $(\forall^c A'\subseteq A)(\forall^c B'\subseteq B) A'\ind[I]_C B'$;
\item[(c)] $(\forall^c A'\subseteq A)(\forall^c B'\subseteq B) A'\ind[Ic]_C \ B'$;
\item[(d)] $(\forall^c A'\subseteq A)(\forall^c B'\subseteq B)(\forall^c C'\subseteq C)(\exists^c D\in[C',C]) A'\ind[I]_D B'$;
\item[(e)] $ (\forall^c C'\subseteq C)(\exists^c D\in[C',C]) A\ind[I]_D B.$
\end{itemize}
By the definition of $\ind[Ic] \ $, (c) and (d) are equivalent.
Whenever $\ind[I]$ has  monotonicity and two-sided countable character, (a) and (b) are equivalent.
Whenever $\ind[I]$ has  monotonicity and two-sided countable character and  $A, B$
are countable, (d) and (e) are equivalent.

Suppose first that $\ind[I]$ is countably based, and that $A$ and $B$ are countable.  It is immediate that  $\ind[I]=\ind[Ic] \ \ $, and that $\ind[I]$ has monotonicity and
two-sided countable character.  Therefore (b) and (c) are equivalent,
so (a) is equivalent to (e) as required.

Now suppose that $\ind[I]$ has monotonicity and two-sided countable character, and (a) is equivalent to (e) whenever
$A, B$ are countable.   Then the parts of (b) and (d) after the quantifiers are equivalent, so (b), and (d) are equivalent.
Therefore (a) is equivalent to (d), so $\ind[I]$ is countably based.
\end{proof}

\begin{prop}  \label{p-preserve}
Let $\ind[I]$ be a countably based ternary relation.
If $\ind[I]$ has any of invariance, transitivity, normality, symmetry, or anti-reflexivity for all countable sets,
then $\ind[I]$ \  has the same property for all small sets.  If $\ind[I]$ has normality, symmetry, and base monotonicity for all countable sets, then $\ind[I] \ $ has base monotonicity for all small sets.
\end{prop}

\begin{proof}  Invariance and symmetry are clear.

Transitivity:  Assume $C\in[D,B], B\ind[I]_C \ A$, and $C\ind[I]_D \ A$.  Let $A_0\subseteq A, B_0\subseteq B$, ${C_0}\subseteq C$,
$D_0\subseteq D$  be countable.  There is a countable $C_1\in[C_0,C]$ such that $B_0\ind[I]_{C_1} A_0$, and a countable $D_1\in [D_0,D]$
such that $C_1\ind[I]_{D_1} A_0$.  By transitivity for countable sets, $B_0\ind[I]_{D_1} A_0$.  This shows that $B\ind[I]_D \ A$.

Normality:  Assume $A\ind[I]_C \ B$.  Let $E_0\subseteq AC, B_0\subseteq B$, $C_0\subseteq C$ be countable.  Let $A_0=E_0\cap A, C_1=C_0\cup(E_0\cap C$).
Then for some countable $C_2\in[C_1,C]$ we have $A_0\ind[I]_{C_2} B_0$.  By normality for countable sets,  $A_0C_2\ind[I]_{C_2} B_0$.
We have
$$E_0=E_0\cap AC=(E_0\cap A)(E_0\cap C)\subseteq A_0C_1\subseteq A_0C_2.$$
Then by monotonicity of $\ind[I]$, $E_0\ind[I]_{C_2} B_0$.  Thus  $AC\ind[I]_C \ B$.

Anti-reflexivity:  Suppose $a\ind[I]_C \ a$.  Let $C_0\subseteq C$ be countable.  For some countable $C_1\in[C_0,C]$ we have
$ a\ind[I]_{C_1}  a$.  Then $ a\in\acl(C_1)$ by the anti-reflexivity of $\ind[I]$, so $ a\in\acl(C)$.

Base monotonicity:  Suppose $C\in[D,B]$ and $A\ind[I]_D \ B$.  We will prove that $A\ind[I]_C \ B$. Let $A_0\subseteq A, B_0\subseteq B$, $C_0\subseteq C$ be countable.
 Let $D_0=C_0\cap D$. We have $B_0C_0\subseteq B$, so there exists a countable $D_1\in[D_0,D]$ such that $A_0\ind[I]_{D_1} B_0C_0$.
By symmetry and normality for countable sets, $A_0\ind[I]_{D_1} B_0 C_0D_1$.  Let $C_1=C_0 D_1$ and $B_1=B_0C_0D_1.$  Then
$A_0\ind[I]_{D_1} B_1$ and $C_1\in[D_1,B_1]$. By base monotonicity for countable sets, $A_0\ind[I]_{C_1} B_1$. By monotonicity,
$A_0\ind[I]_{C_1} B_0$. Therefore $A\ind[I]_C \ B$.
\end{proof}

\begin{prop}  \label{p-countably-local-implies}
Suppose $\ind[I]$ has  monotonicity, base monotonicity, transitivity, symmetry, and countably local character.  Then $\ind[I]\Rightarrow\ind[Ic] \ $.
\end{prop}

\begin{proof}  Suppose  $A\ind[I]_C B$.  Let $A'\subseteq A, B'\subseteq B, C'\subseteq C$ be countable.  By monotonicity, $A'\ind[I]_C B'$.
Countably local character insures that
there is a countable $C_1\subseteq C$ such that $A'\ind[I]_{C_1} C$.  Let $D=C_1 C'$.  Then $D$ is countable and $D\in[C_1,C]$.
By base monotonicity, $A'\ind[I]_{D} C$.  By symmetry, $B'\ind[I]_C A'$ and $C\ind[I]_{D} A'$. By transitivity, $B'\ind[I]_D A'$, and by
symmetry again, $A'\ind[I]_D B'$.  Moreover, $D\in[C',C]$.  This proves that $A\ind[Ic]_C \ B$.
\end{proof}

\begin{cor} \label{c-preserve}
Let $\ind[I]$ be a countable independence relation.
\begin{enumerate}
\item If $\ind[I]$ has countably local character, then $\ind[I]\Rightarrow\ind[Ic] \ \ $.
\item If $\ind[I]\Rightarrow\ind[Ic] \ \ $ then $\ind[Ic] \ \ $ is a countable independence relation.
\end{enumerate}
\end{cor}

\begin{proof}  (1): $\ind[I]$ has symmetry by Remarks \ref{r-fe-vs-ext} (4).  The result follows from symmetry and Proposition \ref{p-countably-local-implies}.

(2): $\ind[Ic]$ \ \ has monotonicity and countable character by Lemma \ref{l-cbased} (3).  By Remark \ref{r-weaker} and Proposition \ref{p-preserve},
$\ind[Ic]$ \ \ satisfies full existence, symmetry, and all the axioms except perhaps extension.  By Remarks \ref{r-fe-vs-ext} (1), extension follows
from full existence, symmetry, and the other axioms, so $\ind[Ic]$ \ \ satisfies extension as well.
\end{proof}

\begin{prop}  \label{p-ind[e]-countable-based}
On any first order theory,  the relation $\ind[e]$  is countably based.
\end{prop}

\begin{proof}
We have already observed that $\ind[e]$ satisfies the basic axioms and has finite character.  It is easily seen that $\ind[e]$
has two sided countable (and even finite) character.

Assume first that $A\ind[e]_C B$ and $C_0$ is a countable subset of $C$.
Since $A$ is countable, there is a countable set $C_1\in[C_0,C]$ such that for any finite $\bar a\in A^{<\BN}$ we have
$A\cap \acl(\bar a C)\subseteq \acl(\bar a C_1)$.
Then for any finite $\bar a\in A^{<\BN}$ we have
$$A\cap\acl(\bar a C_1 B)\subseteq A\cap\acl(\bar a CB)\subseteq\acl(\bar a C)\subseteq\acl(\bar a C_1),$$
so $A\ind[e]_{C_1} B$.

For the other direction, since $A$ is countable, there
is a countable set $C_0\subseteq C$ such that $A\cap \acl(\bar a CB)\subseteq \acl(\bar a C_0 B)$ for every finite $\bar a\in A^{<\BN}$.
Now assume that there is  a countable set $C_1\in[C_0,C]$ such that $A\ind[e]_{C_1} B$.
 Then for every finite $\bar a\in A^{<\BN}$,
$$A\cap\acl(\bar aCB)\subseteq A\cap\acl(\bar a C_0 B)\subseteq A\cap\acl(\bar a C_1 B)\subseteq\acl(\bar a C_1)\subseteq \acl(\bar a C),$$
so $A\ind[e]_C B$.
\end{proof}

\begin{prop}  \label{p-thorn-cbased} In first order or continuous logic,
if there exists a strict countably based independence relation $\ind[I]$  with countably local character, then $\thind$ is countably based.
\end{prop}

\begin{proof}  By Result \ref{f-weakest}, $\thind$ is the weakest strict countable independence relation.
Then $\ind[I] \Rightarrow \thind$, so $\thind$ has countably local character.
By Corollary \ref{c-preserve} (1), $\thind\Rightarrow\ind[\th c] \ \ $. By Corollary \ref{c-preserve} (2),  $\ind[\th c]$ \ \ is a strict countable independence relation on models of $T$ that is
weaker than $\thind$.  Therefore $\ind[\th c] \ =\thind$, so $\thind$ is countably based.
\end{proof}

\begin{prop} \label{p-d-cbased}  In first order or continuous logic, if the dividing independence relation $\ind[d]$
is an independence relation with countably local character,  then $\ind[d]$ is countably based.
\end{prop}

\begin{proof}  We give the proof for continuous logic. Work in the big model $\cu N$ of $T$, and fix countable $A, B$ and small $C$ in $\cu N$.
By Lemma \ref{l-cbased} (3), it is enough to check that
\begin{equation}  \label{e-d-cbased}
A\ind[d]_CB \Leftrightarrow (\forall^c C'\subseteq C)(\exists^c D\in [C',C])A\ind[d]_DB.
\end{equation}

$\Rightarrow$:  Suppose that $A\ind[d]_CB$.  Since $\ind[d]$ is an independence relation with countably local character,
we have that $\ind[d]\Rightarrow \ind[dc] \ \ $ by Corollary \ref{c-preserve}, whence we get the forward implication of (\ref{e-d-cbased}).

$\Leftarrow$:  Suppose that $A\nind[d]_CB$.  Then for some $\vec a\in A^{<\BN}$ and some  formula $\Phi(\vec x,B,C)$,
$\cu N\models \Phi(\vec a,B,C)=0$ and $\Phi(\vec x,B,C)$ divides over $C$.  Take a countable (even finite) $C'\subseteq C$ such
that $\Phi(\vec x,B,C)=\Phi(\vec x,B ,C')$.   Then for any countable $D\in [C',C]$, $\Phi(\vec x,B,C)$ divides over $D$,
so $A\nind[d]_D B$ and the right hand side of (\ref{e-d-cbased}) fails.
\end{proof}

In the paper [Be3],  Ben Yaacov defined simple continuous theories and showed that they satisfy the hypotheses of
Proposition \ref{p-d-cbased}.  Thus on models of a simple theory, $\ind[d]$ is countably based.

It can  be also shown that for any first order or continuous theory, $\ind[a]$ and $\ind[M]\ $ are countably based.
These results will not be needed in this paper, and will be left to the reader.

\subsection{Countable Union Property}

\begin{df} A ternary relation $\ind[I]$ has the \emph{countable union property} if whenever
$A, B, C$ are countable, $C=\bigcup_n C_n$, and $C_n\subseteq C_{n+1}$ and $A\ind[I]_{C_n} B$
for each $n$, we have $A\ind[I]_C B$.
\end{df}

\begin{rmk}  \label{r-union-cbased} If $\ind[I]$ has monotonicity, then $\ind[I]$ has the countable union property if and only $\ind[Ic] \ $ has the countable union property,
because $\ind[I]$and $\ind[Ic] \ $ agree on countable sets.
\end{rmk}

Given two ternary relations $\ind[I]$ and $\ind[J]$ over $\cu N$, $\ind[I]\wedge\ind[J]$ will denote the relation $\ind[K] \ $ such that
$$A\ind[K]_C \ B\Leftrightarrow A\ind[I]_C B\wedge A\ind[J]_C B.$$

\begin{prop}  \label{p-union-pair}  Suppose $\ind[I]$ and $\ind[J]$ are both countably based and have the countable union property.
Then the relation $\ind[I]\wedge\ind[J]$ is also countably based.
\end{prop}

\begin{proof}   Let $A, B$ be countable and let $\ind[K] \ =\ind[I]\wedge\ind[J]$.  By  Lemma \ref{l-cbased} (3), it is enough to show that
$$ A\ind[K]_C \ B \Leftrightarrow (\forall^c C'\subseteq C)(\exists^c D\in[C',C]) A\ind[K]_D \ B.$$
The implication from right to left is trivial. For the other direction, assume $A\ind[K]_C B$ and let $C'\subseteq C$ be countable.
Since both $\ind[I]$ and $\ind[J]$ are countably based, there is a  sequence $\< D_n\>_{n\in\BN}$ of countable sets such that
$D_n\subseteq D_{n+1}$ and $D_n\in[C',C]$ for each $n\in\BN$, $A\ind[I]_{D_n} B$ for each even $n$, and  $A\ind[J]_{D_n} B$ for each odd $n$.
Let $D=\bigcup_n D_n$. Then $D\in[C',C]$ and $D$ is countable.  Since both $\ind[I]$ and $\ind[J]$ have the countable union property, we have $A\ind[K]_D B$,
as required.
\end{proof}

\begin{prop} \label{p-preserve-finitechar}  If $\ind[I]$ has monotonicity,
finite character, and the countable union property, then $\ind[Ic] \ $ has finite character.
\end{prop}

\begin{proof}
Suppose $A'\ind[Ic]_C \ B$ for every finite $A'\subseteq A$.  Let
$A_0\subseteq A, B_0\subseteq B$, ${C_0}\subseteq C$ be countable.  Let $A_0=\bigcup_{n}E_n$ where
$E_n$ is finite and $E_n\subseteq E_{n+1}$ for each $n$.  By induction on $n$, there is a sequence
of countable sets $\<D_n\>_{n\in\BN}$ such that for each $n$, $D_n\in[C_0,C]$, $D_n\subseteq D_{n+1}$, and $E_n\ind[I]_{D_n} B_0$.
By monotonicity, $E_n\ind[I]_{D_k} B_0$ whenever $n\le k$.  Let $D=\bigcup_n D_n$.  Then $D$ is countable and $D\in[C_0,C]$.
By the countable union property, $E_n\ind[I]_D B_0$ for each $n$.  Hence by monotonicity and finite character for $\ind[I]$,
we have $A_0\ind[I]_D B_0$.  This shows that $A\ind[Ic]_C \ B$, so $\ind[Ic] \ $ has finite character.
\end{proof}

\begin{lemma}  \label{l-d-union}
In first order or continuous logic, the dividing independence relation $\ind[d]$ has the countable union property.
\end{lemma}

\begin{proof}  We give the proof for continuous logic.
Suppose $A, B, C$ are countable, $C=\bigcup_n C_n$, and $C_n\subseteq C_{n+1}$ and $A\ind[d]_{C_n} B$ for each $n$,
but $A\nind[d]_C B$.  Then there exists $\bar a\in A^{<\BN}$ and a continuous formula $\Phi(\bar x,B,C)$ such that
$\Phi(\bar a,B,C)=0$ and $\Phi(\bar x, B,C)$ divides over $C$.  Then $\Phi(\bar x,B,C)=\Phi(\bar x,B,C_n)$ for some $n\in\BN$.
Hence $\Phi(\bar a,B,C_n)=0$ and $\Phi(\bar x,B,C_n)$ divides over $C_n$, contradicting $A\ind[d]_{C_n} B$.
\end{proof}

\begin{lemma}  \label{l-ind[e]-union}
On a first order theory,  the relation $\ind[e]$ has the countable union property.
\end{lemma}

\begin{proof}  Suppose $A, B, C$ are countable, $C=\bigcup_n C_n$, and for all $n$, $C_n\subseteq C_{n+1}$ and $A\ind[e]_{C_n} B$.
Let $\bar a\in A^{<\BN}$ and $g\in A\cap\acl(\bar a BC)$.  Then $g\in A\cap\acl(\bar a BC_n)$ for some $n$.  Since $A\ind[e]_{C_n} B$,
$g\in \acl(\bar a C_n)$.  Hence $g\in \acl(\bar a C)$ and $A\ind[e]_C B$.
\end{proof}

It can also be shown that for any first order theory, the relations $\ind[a], \ind[M] \ ,$ and $\thind$ have the countable union property.
We will not need those results, and leave the proofs as exercises for the reader.

\section{Randomizations}

\subsection{The Theory $T^R$}

Assume hereafter that the models of $T$ have at least two elements.  We now introduce the continuous theory $T^R$.

The \emph{randomization signature} $L^R$ is the two-sorted continuous signature
with sorts $\BK$ (for random elements) and $\BB$ (for events), an $n$-ary
function symbol $\l\varphi(\cdot)\rr$ of sort $\BK^n\to\BB$
for each first order formula $\varphi$ of $L$ with $n$ free variables,
a $[0,1]$-valued unary predicate symbol $\mu$ of sort $\BB$ for probability, and
the Boolean operations $\top,\bot,\sqcap, \sqcup,\neg$ of sort $\BB$.  The signature
$L^R$ also has distance predicates $d_\BB$ of sort $\BB$ and $d_\BK$ of sort $\BK$.
In $L^R$, we use ${\sa B},{\sa C},\ldots$ for variables or parameters of sort $\BB$. ${\sa B}\doteq{\sa C}$
means $d_\BB({\sa B},{\sa C})=0$, and ${\sa B}\sqsubseteq{\sa C}$ means ${\sa B}\doteq{\sa B}\sqcap{\sa C}$.

A pre-structure for $T^R$ will be a pair $\cu P=(\cu K,\cu E)$ where $\cu K$ is the part of sort $\BK$ and
$\cu E$ is the part of sort $\BB$.\footnote{In [BK], the set of events was denoted by $\cu B$, but we use $\cu E$ here
to reserve the letters $\cu A, \cu B, \cu C$ for subsets of $\cu E$.}
The \emph{reduction} of $\cu P$ is the pre-structure $\cu N=(\hat{\cu K},\hat{\cu E})$ obtained from
$\cu P$ by identifying elements at distance zero in the metrics $d_\BK$ and $d_\BB$,
 and the associated mapping from $\cu P$ onto $\cu N$ is called the \emph{reduction map}.
The \emph{completion} of $\cu P$ is the structure obtained by completing the metrics in the reduction of $\cu P$.
By a \emph{pre-complete-structure} we mean a pre-structure $\cu P$ such that the reduction of $\cu P$ is equal to the completion of $\cu P$.
By a \emph{pre-complete-model} of $T^R$ we mean a pre-complete-structure that is a pre-model of $T^R$.

In [BK], the randomization theory $T^R$ is defined by listing a set of axioms.
We will not repeat these axioms here, because it is simpler to give the following model-theoretic
characterization of $T^R$.

\begin{df}  \label{d-neat}
Given a model $\cu M$ of $T$, a \emph{neat randomization of} $\cu M$ is a pre-complete-structure $\cu P=(\cu L,\cu F)$ for $L^R$
equipped with an atomless probability space $(\Omega,\cu F,\mu )$ such that:
\begin{enumerate}
\item $\cu F$ is a $\sigma$-algebra with $\top,\bot,\sqcap, \sqcup,\neg$ interpreted by $\Omega,\emptyset,\cap,\cup,\setminus$.
\item $\cu L$ is a set of functions $a\colon\Omega\to M$.
\item For each formula $\psi(\bar{x})$ of $L$ and tuple
$\bar{a}$ in $\cu L$, we have
$$\l\psi(\bar{a})\rr=\{\omega\in\Omega:\cu M\models\psi(\bar{a}(\omega))\}\in\cu F.$$
\item $\cu F$ is equal to the set of all events
$ \l\psi(\bar{a})\rr$
where $\psi(\bar{v})$ is a formula of $L$ and $\bar{a}$ is a tuple in $\cu L$.
\item For each formula $\theta(u, \bar{v})$
of $L$ and tuple $\bar{b}$ in $\cu L$, there exists $a\in\cu L$ such that
$$ \l \theta(a,\bar{b})\rr=\l(\exists u\,\theta)(\bar{b})\rr.$$
\item On $\cu L$, the distance predicate $d_\BK$ defines the pseudo-metric
$$d_\BK(a,b)= \mu \l a\neq b\rr .$$
\item On $\cu F$, the distance predicate $d_\BB$ defines the pseudo-metric
$$d_\BB({\sa B},{\sa C})=\mu ( {\sa B}\triangle {\sa C}).$$
\end{enumerate}
\end{df}

Note that if $\cu H\prec\cu M$, then every neat randomization of $\cu H$ is also a neat randomization of $\cu M$.

\begin{df}  For each first order theory $T$, the \emph{randomization theory} $T^R$ is
the set of sentences that are true in all neat randomizations of models of $T$.
\end{df}

It follows that for each first order sentence $\varphi$, if $T\models\varphi$
then  $T^R\models \l\varphi\rr\doteq \top$.

\begin{result}  \label{f-perfectwitnesses}  (Fullness, Proposition 2.7 in [BK]).
Every pre-complete-model $\cu N=(\cu K,\cu E)$ of $T^R$ has perfect witnesses, i.e.,
\begin{enumerate}
\item  For each first order formula $\theta(u,\bar v)$ and each $\bar{b }$ in $\cu K^n$ there exists $a \in\cu K$ such that
$$ \l\theta(a,\bar b)\rr \doteq
\l(\exists u\,\theta)(\bar{b })\rr;$$
\item For each ${\sa E}\in\cu E$ there exist $a ,b \in\cu K$ such that
${\sa E}\doteq\l a=b \rr$.
\end{enumerate}
\end{result}

The following results are proved in [Ke1], and are stated in the continuous setting in [BK].

\begin{result}  \label{f-complete} (Theorem 3.10 in [Ke1], and Theorem 2.1 in [BK]).
For every complete first order theory $T$, the randomization theory $T^R$ is complete.
\end{result}

\begin{result} \label{f-qe} (Strong quantifier elimination, Theorems 3.6 and 5.1 in [Ke1], and Theorem 2.9 in [BK])
Every formula $\Phi$ in the continuous language $L^R$
is $T^R$-equivalent to a formula with the same free variables
and no quantifiers of sort $\BK$ or $\BB$.
\end{result}

\begin{result}  \label{f-T^R} (Proposition 4.3 and Example 4.11 in [Ke1], and Proposition 2.2 and Example 3.4 (ii) in [BK]).
Every model $\cu M$ of $T$ has neat randomizations.
\end{result}

\begin{result}  \label{f-glue}  (Lemma 2.1.8 in [AGK])
Let $\cu P=(\cu K,\cu E)$ be a pre-complete-model of $T^R$ and let $a ,b \in\cu K$ and ${\sa B}\in\cu E$.
Then there is an element $c \in\cu K$ that agrees with $a $ on ${\sa B}$ and agrees with $b $ on $\neg{\sa B}$,
that is, ${\sa B}\sqsubseteq\l c =a \rr$ and $(\neg{\sa B})\sqsubseteq\l c =b \rr$.
\end{result}

\begin{result}  \label{f-representation}  (Proposition 2.1.10 in [AGK])
Every model of $T^R$ is isomorphic to the reduction of a neat randomization $\cu P$ of a model of $T$.
\end{result}

\subsection{Permanent Assumptions}

From now on we will work within the big model $\cu N=( {\cu K}, {\cu E})$ of $T^R$.
We let $\cu M$ be the big model of $T$ and let $\cu P=(\cu L,\cu F)$ be a neat randomization of $\cu M$
with probability space $(\Omega,\cu E,\mu)$, such that $\cu N$ is the reduction of $\cu P$.
We may further assume that the probability space $(\Omega,\cu F, \mu)$ of $\cu P$ is complete (that is, every set
that contains a set of $\mu$-measure one belongs to $\cu F$), and that every
function $a\colon \Omega\to M$ that agrees with some $b\in \cu L$ except on a $\mu$-null subset of $\Omega$ belongs to $\cu L$.
The existence of $\cu P$ is guaranteed by Result \ref{f-representation} (Proposition 2.1.10 in [AGK]), and the further assumption
follows from the proof in [AGK].

By saturation, $\cu K$ and $\cu E$ are large.  Hereafter, $A, B, C$ will always denote small subsets of $\cu K$.
For each element $\bo a\in {\cu K}$, we will also choose once and for all an element $a\in\cu L$ such that
the image of $a$ under the reduction map is $\bo a$.  It follows that for each first order formula $\varphi(\bar v)$,
$\l\varphi(\bar{\bo a})\rr$ in $\cu N$ is the image of $\l\varphi(\bar a)\rr$ in $\cu P$ under the reduction map.

For any small $A\subseteq\cu K$ and each $\omega\in\Omega$, we define
$$ A(\omega)=\{a(\omega)\mid \bo a\in  A\},$$
and let $\cl( A)$ denote the closure of $ A$ in the metric $d_\BK$.  When $\cu A\subseteq {\cu E}$,
$\cl(\cu A)$ denotes the closure of $\cu A$ in the metric $d_\BB$, and
$\sigma(\cu A)$ denotes the smallest $\sigma$-subalgebra of $ {\cu E}$ containing $\cu A$.
Since the cardinality $\upsilon$ of $\cu N$ is inaccessible, whenever $A\subseteq\cu K$ is small, the closure $\cl(A)$ and
the set of $n$-types over $A$ is small.  Also, whenever $\cu A\subseteq\cu E$ is small, the closure $\cl(\cu A)$ is small.

\subsection{Definability in $T^R$}

As explained in [AGK, Remark 2.2.4], in models of $T^R$ we need only consider definability over sets of parameters of sort $\BK$.

We write $\dcl_\BB( A)$ for the set of elements of sort $\BB$ that are definable over $ A$ in $\cu N$,
and write $\dcl( A)$ for the set of elements of sort $\BK$ that are definable over $ A$ in $\cu N$.
Similarly for $\acl_\BB( A)$ and $\acl( A)$.  We often use the following result without explicit mention.

\begin{result}  \label{f-acl=dcl}  ([AGK], Proposition 3.3.7; see also [Be2], Theorem 2.17 (iv) and Corollary 5.9.)
$\acl_\BB( A)=\dcl_\BB( A)$ and $\acl( A)=\dcl( A)$.
\end{result}

\begin{df}  We say that an event $\sa E$ is \emph{first order definable over $A$}, in symbols $\sa E\in\fo_\BB(A)$, if
$\sa E=\l\theta(\bar{\bo a})\rr$ for some formula $\theta$ of $L$ and some tuple $\bar{\bo a}\in A^{<\BN}$.
\end{df}

\begin{df}
We say that $\bo b$ is \emph{first order definable over $ A$}, in symbols $\bo b\in\fo( A)$, if there is a functional formula
$\varphi(u,\bar v)$ and a tuple $\bar{\bo a}\in {A}^{<\BN}$ such that
$\l \varphi(b,\bar{a})\rr=\top$.
\end{df}

Note that the formula $u=v$ is functional, and hence $A\subseteq\fo(A)$.

\begin{result} \label{f-separable}  ([AGK], Theorems 3.1.2 and  3.3.6)
$$\dcl_\BB( A)=\cl(\fo_\BB( A))=\sigma(\fo_\BB(A))\subseteq\cu E,\qquad \dcl( A)=\cl(\fo( A))\subseteq\cu K.$$
If $A$ is empty, then $\dcl_\BB(A)=\{\top,\bot\}$, and $\fo(A)$ is closed, so $\dcl(A)=\fo(A)$.
\end{result}

It follows that whenever $A$ is small, $\dcl(A)$ and $\dcl_\BB(A)$ are small.

Using the distance function between an element and a set, \ref{f-separable} can be re-stated as follows:

\begin{rmk}  \label{r-distance}
$$\sa E\in\dcl_\BB(A) \mbox{ if and only if } d_\BB(\sa E,\fo_\BB(A))=0,$$
and
$$\bo b\in\dcl(A) \mbox{ if and only if } d_\BK(\bo b,\fo(A))=0.$$
\end{rmk}

\begin{rmk}  \label{r-dcl-B}  For each small $A$,
$$\fo_\BB(\fo(A))=\fo_\BB(A),\quad\dcl_\BB(\dcl(A))=\dcl_\BB(A).$$
\end{rmk}

\begin{proof}
The first equation is clear, and the second equation follows from the first equation and Remark \ref{r-distance}.
\end{proof}

We will sometimes use the $\l \ldots\rr$ notation in a general setting.  Given a property $P(\omega)$, we write
$$\l P\rr=\{\omega\in\Omega\,:\,P(\omega)\},$$
and we say that
$$P(\omega) \mbox{ holds }\as.$$
if $\l P\rr$ contains a set  $\sa A\in\cu F$ such that $\mu(\sa A)=1$.
For example, $\l b\in\dcl^{\cu M}(A)\rr$ is the set of all $\omega\in\Omega$ such that $b(\omega)\in\dcl^{\cu M}(A(\omega))$.

\begin{result}  \label{f-pointwisemeasurable}  ([AGK], Lemma 3.2.5)
If $A$ is countable, then
$$\l b\in\dcl^{\cu M}(A)\rr= \bigcup\{\l\theta(b,\bar a)\rr\mid\theta(u,\bar v) \mbox{ functional, } \bar{\bo a}\in A^{<\BN}\},$$
and  $\l b\in\dcl^{\cu M}(A)\rr\in\cu F$.
\end{result}

It follows that for each countable $A$,
$$b(\omega)\in\dcl(A(\omega))\as.\Leftrightarrow\mu(\l b\in\dcl(A)\rr)=1.$$

\begin{df} \label{d-pointwise-def}
 We say that $\bo b$ is \emph{pointwise definable over $A$}, in symbols $\bo b\in\dcl^\omega(A)$, if
$$\mu(\l b\in\dcl^{\cu M}(A_0)\rr)=1$$
for some countable $A_0\subseteq A$.

We say that $\bo b$ is \emph{pointwise algebraic over $A$}, in symbols $\bo b\in\acl^\omega(A)$, if
$$\mu(\l b\in\acl^{\cu M}(A_0)\rr)=1$$
for some countable $A_0\subseteq A$.
\end{df}

\begin{rmk}  \label{r-dcl-omega}
$\dcl^\omega$ and $\acl^\omega$ have countable character, that is,
$\bo b\in\dcl^\omega(A)$ if and only if $\bo b\in\dcl^\omega(A_0)$ for some countable $A_0\subseteq A$, and similarly for $\acl^\omega$.
\end{rmk}

\begin{result}  \label{f-dcl3}  ([AGK], Corollary 3.3.4) For any element $\bo b\in {\cu K}$,
$\bo{b}$ is definable over $ A$ if and only if:
\begin{enumerate}
\item $\bo b$ is pointwise definable over $A$;
\item for each functional formula $\varphi(u,\bar v)$ and tuple $\bar{\bo a}\in A^{<\BN}$, $\l\varphi(\bo b, \bar{\bo a})\rr$ is definable over $A$.
\end{enumerate}
\end{result}

\begin{cor}  \label{c-pointwise-alg-def}
In $\cu N$ we always have
$$\acl(A)=\dcl(A)\subseteq\dcl^\omega(A)=\dcl^\omega(\dcl^\omega(A))\subseteq\acl^\omega(A)=\acl^\omega(\acl^\omega(A)).$$
\end{cor}

The following proposition gives a warning: the set $\dcl^\omega(A)$ is almost always large
(By contrast, by Result \ref{f-algebraic-cardinality}, for any small set $A$, $\acl(A)$ and $\dcl(A)$ are small.)

\begin{prop}  \label{p-large}  If $|A|>1$, or even if $|\dcl^\omega(A)|>1$, then $\dcl^\omega(A)$ is large.
\end{prop}

\begin{proof}  We may assume that $A$ is countable.  Take two elements $\bo a\ne \bo b\in\dcl^\omega(A)$.  Then $\mu(\l a\ne b\rr)=r>0$.
By Result \ref{f-glue}, for each event $\sa E\in\cu E$, there is an element $\bo c_{\sa E}\in\cu K$ that agrees with $\bo a$ on $\sa E$ and agrees with $\bo b$ elsewhere.
Then
$\mu(\l \bo c_{\sa E}\in \dcl(A)\rr)=1$, so $\bo c_{\sa E}\in\dcl^\omega(A)$. For each $n$ there is a set of $n$ events $\sa E_1,\ldots,\sa E_n$
such that $d_\BK({\bo c}_{\sa E_i},{\bo c}_{\sa E_j})=d_\BB(\sa E_i,\sa E_j)=r/2$ whenever $i<j\le n$.  Then by saturation, the set
$\dcl^\omega(A)$ has cardinality $\ge\upsilon$ and hence is large.
\end{proof}

\begin{cor}  \label{c-subset}  Let $A$ be a countable subset of $\cu K$ with $|A|>1$.  Then the set of all small $B$ such that
$B(\omega)\subseteq A(\omega) \as.$ is large and contains every subset of $A$.
\end{cor}

\begin{proof}  Similar to the proof of Proposition \ref{p-large}.
\end{proof}

\subsection{A Downward Result}

\begin{df}  We say that $T$ has $\acl=\dcl$ if for every set $A$ in $\cu M$ we have  $\acl^{\cu M}(A)=\dcl^{\cu M}(A)$.
\end{df}

For example, every theory with a definable linear ordering has $\acl=\dcl$.  In this section we prove that if $T$ has $\acl=\dcl$ and $T^R$ is real rosy, then $T$ is real rosy.

\begin{lemma}  \label{l-big-embedding}
\noindent\begin{itemize}
\item[(1)]   There is a mapping $b\mapsto \ti b$ from $M$ into $\cu K$ such that
for each tuple $\bar a$ in $M$ and first order formula $\varphi(\bar v)$, if $\cu M\models\varphi(\bar a)$ then $\mu(\l \varphi(\ti {\bar a})\rr) =1$.
\item[(2)]  Let $\ti B =\{\ti b\mid b\in B\}$.  If $B=\dcl^{\cu M}(A)$ then $\ti B=\fo(\ti A)=\dcl(\ti A)$,
\end{itemize}
\end{lemma}

\begin{proof} (1)  The result is trivial if $M$ is finite, so we assume $M$ has cardinality $\upsilon$.  Let $M=\{a_\alpha\mid \alpha <\upsilon\}$ with no repetitions, and for each $\beta<\upsilon$ let $M_\beta=\{a_\alpha\mid \alpha<\beta\}.$  Let $S(\beta)$ be the statement
\begin{itemize}
\item[$S(\beta)$:] For each tuple $\bar a$ in $M_\beta$ and each $\varphi(\bar v)$, if $\cu M\models\varphi(\bar a)$ then $\mu(\l \varphi(\ti {\bar a})\rr) =1$.
\end{itemize}
By transfinite recursion we will build a sequence
$\<\ti{a}_\alpha\mid \alpha <\upsilon\>$ of elements of $\cu K$ such that $S(\alpha+1)$ holds for each $\alpha<\upsilon$,

Suppose we have already defined $\<\ti{a}_\alpha\mid \alpha<\beta\>$  such that $S(\alpha+1)$ holds for all $\alpha<\beta$.  Since each tuple in $M_\beta$ is a tuple in $M_{\alpha_+1}$ for some $\alpha<\beta$, $S(\beta)$ holds.  Let $\Gamma(u)$ be the set of all continuous statements
$$\Gamma(u)=\{\mu(\l \varphi(\ti{\bar a},u)\rr) =1\mid \bar{a}\in (M_\beta)^{<\BN}, \varphi(\bar v,u)\in [L],\cu M\models \varphi(\bar{a},a_\beta)\}.$$
$\Gamma(u)$ is small, and up to equivalence, $\Gamma(u)$ is closed under finite conjunctions.  If $\bar{a}$ is a tuple in $M_\beta$ and $\cu{M}\models \varphi(\bar{a},a_\beta)$, then $\cu{M}\models\exists u \varphi(\bar{a},u)$, so by $S(\beta)$ we have $\mu(\l \exists u\varphi(\ti {\bar a},u)\rr) =1.$  By Fullness, $\Gamma(u)$ is finitely satisfiable in $\cu N$.  Since $\cu N$ is saturated, $\Gamma(u)$ is satisfiable in $\cu N.$  We choose an element $\bo b\in\cu K$ that satisfies $\Gamma(u)$ in $\cu N$ and define $\ti{a}_\beta = \bo b$.  Then $S(\beta+1)$ holds.  This completes our transfinite recursion and proves (1).

(2)  For any first order functional formula $\psi(X,y)$, we have
$\cu M\models \psi(A,b)$ iff $\cu N\models\l\psi(\ti A,\ti b)\rr=\top$, and $\cu N\models\l(\exists ^{\le 1} y)\psi(\ti A,y)\rr=\top$.
Therefore $\ti B=\fo(\ti A)$.  For any two distinct elements $\ti b, \ti c$ of $\ti B$, we have $\cu M\models b\ne c$, so $\cu N\models\mu(\l\ti b\ne\ti c\rr)=1$
and hence $d_\BK(\ti b,\ti c)=1$.  Thus any two elements of $\ti B$ have distance $1$ from each other,
so $\ti B$ is closed in $\cu N$.  By Result \ref{f-separable}, $\fo(\ti A)=\cl(\fo(\ti A))=\dcl(\ti A)$.
\end{proof}

\begin{prop}  \label{p-downward}  Suppose $T$ has $\acl=\dcl$,
 $\ind[I]$ is one of the relations $\ind[a],\, \ind[M], \,\thind$, and $A, B, C$ are small subsets of $M$.
 If $\ti A\ind[I]_{\ti C} \ti B$ holds in $\cu N$, then $A\ind[I]_C B$ holds in $\cu M$.
\end{prop}

\begin{proof}  Suppose first that $\ti A\ind[a]_{\ti C} \ti B$ in $\cu N$. Then
$$\acl(\ti A \ti C)\cap\acl(\ti B\ti C)=\acl(\ti C).$$
Let
$$A'=\acl^{\cu M}(AC), B'=\acl^{\cu M}(BC), C'=\acl^{\cu M}(C).$$
By Result \ref{f-acl=dcl} and Lemma \ref{l-big-embedding},
 $A'=\dcl^{\cu M}(AC)$, and $\ti{A'}=\dcl(\ti A\ti C)=\acl(\ti A\ti C)$.  Similarly for $B'$ and $C'$.
Therefore $\ti {A'}\cap \ti {B'}= \ti {C'}$.  It follows that $A'\cap B'=C'$, which means that $A\ind[a]_C B$ in $\cu M$.

Now suppose that $\ti A\ind[M]_{\ti C} \ti B$ in $\cu N$.  Let $D\in[C,\acl^{\cu M}(BC)]$ in $\cu M$.  Then $D\subseteq\dcl^{\cu M}(BC)$,
so by Lemma \ref{l-big-embedding}, $\ti D\subseteq\dcl(\ti B\ti C)=\acl(\ti B\ti C)$.  Hence $\ti A\ind[a]_{\ti D} \ti B$ in $\cu N$, so
$A\ind[a]_D B$  and $A\ind[M]_C B$ in $\cu M$.

Suppose that $\ti A\thind_{\ti C}\ti B$ in $\cu N$.  Let $B\subseteq D$ with $D$ small in $\cu M$.  Then $\ti B\subseteq \ti D$ in $\cu N$ and $\ti D$ is small.
Hence there exists $A_1\equiv_{\ti B\ti C} \ti A$ such that $A_1\ind[M]_{\ti C}\ti D$ in $\cu N$.  Then for every $E_1\in[\ti C,\acl(\ti C\ti D)]$
we have $A_1\ind[a]_{E_1}\ti D$ in $\cu N$.  By Lemma \ref{l-big-embedding} and the assumption that $T$ has $\acl=\dcl$, $E_1\in[\ti C,\acl(\ti C\ti D)]$ if and only if
 $E_1=\ti E$ for some $E\in[C,\acl(CD)]$.  So $A_1\ind[a]_{\ti E}\ti D$ in $\cu N$ for every $E\in[C,\acl(CD)]$.
Let $F=\acl(CD)=\dcl(CD)$ in $\cu M$, so  by Lemma \ref{l-big-embedding}, $\ti F=\acl(\ti C\ti D)$ in $\cu N$.
By saturation in $\cu M$, there is a set $G\subseteq M$ such that for every first order formula $\varphi(X,\ti F)$ such that $\l\varphi(A_1,\ti F)\rr=\top$
in $\cu N$, we have $\cu M\models \varphi(G,F)$.  Then in $\cu N$ we have $\ti G\equiv_{\ti B\ti C} A_1\equiv_{\ti B\ti C} \ti A$,
and $\ti G\ind[a]_{\ti E}\ti D$  for every $E\in[C,\acl(CD)]$. Therefore in $\cu M$ we have $G\equiv_{BC} A$ and $G\ind[a]_E D$  for every $E\in[C,\acl(CD)]$.
It follows that $G\ind[M]_C D$ and $A\thind_C B$ in $\cu M$.
\end{proof}

The following Corollary can be compared with Corollary 7.9 of [EG], which says that if $T^R$ is maximally real rosy then $T$ is real rosy.

\begin{cor}  \label{c-rosy-downward}
Suppose $T$ has $\acl=\dcl$ and $T^R$ is real rosy.  Then $T$ is real rosy.
\end{cor}

\begin{proof}  By hypothesis, the relation $\thind$ over $\cu N$ has local character, with some bound $\kappa^{\cu N}(A)$.
We show that $\thind$ over $\cu M$ has local character with bound $\kappa^{\cu M}(A)=\kappa^{\cu N}(\ti A)$.
Let $A, B\subseteq M$ be small.  Then there is a set $C'\subseteq \ti B$ such that $|C'|<\kappa^{\cu N}(\ti A)$ and $\ti A\thind_{C'} \ti B$ in $\cu N$.
It is  clear that $C'=\ti C$ for some set $C\subseteq B$.  By Proposition \ref{p-downward} we have $A\thind_C B$ in $\cu M$.
\end{proof}

\section{Independence in the Event Sort}

The one-sorted continuous theory $\APr$ of atomless probability algebras is studied in the papers [Be1], [Be2], and [BBHU].
By Fact 2.10 in [Be2], for every model $\cu N=(\cu K,\cu E)$ of $T^R$, the event sort $(\cu E,\mu)$ of $\cu N$
 is a model of $\APr$.  For each cardinal $\kappa$, if $\cu N$ is $\kappa$-saturated then $( {\cu E},\mu)$ is $\kappa$-saturated.
For every set $\cu A\subseteq\cu E$, we have
$\acl(\cu A)=\dcl(\cu A)=\sigma(\cu A)$ in $(\cu E,\mu)$.  The algebraic independence relation in $(\cu E,\mu)$ is the relation
$$ \cu A\ind[a]_{\cu C} \cu B \Leftrightarrow \sigma(\cu A\cu C)\cap\sigma(\cu B\cu C)=\sigma(\cu C).$$

\begin{result}  \label{f-probability-algebra}  (Ben Yaacov [Be1])
The theory $\APr$ is separably categorical, admits quantifier elimination, and is stable.  The independence relation $\ind[d]$ on $\APr$
is the same as the relation of probabilistic independence, given by $\cu A\ind[d]_{\cu C} \cu B$ iff
$$\mu[\sa A\sqcap\sa B | \sigma(\cu C)]=\mu[\sa A | \sigma(\cu C)]\mu[\sa B | \sigma(\cu C)] \as \mbox{ for all } \sa A\in \sigma(\cu A), \sa B\in\sigma(\cu B).$$
\end{result}

\subsection{Dividing Independence in the Event Sort}

We now consider the analogue of the dividing independence relation $\ind[d]$ in the event sort.
Given sets $C,D\subseteq \cu K$, it will be convenient to introduce the notation  $\cu D=\fo_\BB(D)$ and $\cu D_C=\fo_\BB(DC)$.

\begin{rmk}  \label{r-ind-B}
 By Result \ref{f-separable}, $\acl_\BB(D)=\sigma(\cu D)=\acl(\cu D)$.
\end{rmk}

\begin{df}  For small $A, B, C\subseteq\cu K$, define
$$  A\ind[d\BB]_C \ B \Leftrightarrow \cu A_C\ind[d]_{\cu C} \cu B_C \mbox{ in } (\cu E, \mu).$$
\end{df}

\begin{lemma}  \label{l-fB-basic}
$\ind[d\BB] \ \ $ satisfies the basic axioms, symmetry, finite character, and the countable  union property.
\end{lemma}

\begin{proof}  By Result \ref{f-probability-algebra}, $\ind[d] \ $ is
an independence relation over $(\cu E,\mu)$.  It follows easily that $\ind[d\BB] \ \ $ satisfies invariance, monotonicity,
base monotonicity, normality, finite character, and symmetry.

Transitivity: Suppose $C\in[D,B]$, $B\ind[d\BB]_C \ A$, and $C\ind[d\BB]_D \ A$.
Then $\cu A_C$,  $\cu B_C$, and $\cu C$ are small, $\cu C\in[\cu D,\cu B]$,
and $\cu A_D\subseteq\cu A_C$.  We have
$\cu B_C\ind[d]_{\cu C} \cu A_C$ and $\cu C_D\ind[d]_{\cu D} \cu A_D$.  By monotonicity of $\ind[d]$, $\cu B_D\ind[d]_{\cu C} \cu A_D$.  Then
by transitivity of $\ind[d]$,  $\cu B_D\ind[d]_{\cu D} \cu A_D$, so $B\ind[d\BB]_D \ A$.

Countable union property:
By Lemma \ref{l-d-union},
$\ind[d]$ has the countable union property over $(\cu E,\mu)$.  Suppose $A,B, C$ are countable, $C=\bigcup_n C_n$, and $C_n\subseteq C_{n+1}$
and $A\ind[d\BB]_{C_n} B$ for each $n$.  By monotonicity for $\ind[d]$, whenever $n\le m$ we have $\cu A_{C_n}\ind[d]_{\cu C_m} \cu B_{C_n}$.
By the  countable union property for $\ind[d]$ over $(\cu E,\mu)$, for each $n$ we have $\cu A_{C_n}\ind[d]_{\cu C} \cu B_{C_n}$.
Then by finite character and monotonicity for $\ind[d]$, it follows that $\cu A_{C}\ind[d]_{\cu C} \cu B_{C}$, so $A\ind[d\BB]_C \ B$, so $\ind[d\BB] \ \ $ has
the countable union property.
\end{proof}

\begin{lemma}  \label{l-fB-existence}
$\ind[d\BB] \ \ $ satisfies extension and full existence.
\end{lemma}

\begin{proof} By Remarks \ref{r-fe-vs-ext} (1), it suffices to prove that $\ind[d\BB] \ \ $ satisfies  full existence.
For any $A,B,C$, we must show that there exists $A'\equiv_C A$ such that $A'\ind[d\BB]_C B$.
We may assume that $C\subseteq A$, so that $\cu A=\cu A_C$.  Since $\ind[d]$ has full existence in $(\cu E,\mu)$,
there exists $\cu A'\subseteq\cu E$ such that $\cu A'\equiv_{\cu C} \cu A$ and $\cu A'\ind[d]_{\cu C} \cu B_C$ in $(\cu E,\mu)$.
Note that every quantifier-free formula of $T^R$ with parameters in $\cu A\cup C$ has the form $f(\mu(\tau_1),\ldots,\mu(\tau_m))$
where $f\,:\,[0,1]^m\to[0,1]$ is continuous, and each $\tau_i$ is a Boolean term of the form
$$\tau_i(\l\theta_1(C)\rr,\ldots,\l\theta_n(C)\rr,\sa A_1,\ldots,\sa A_k)$$
with $\sa A_1,\ldots\sa A_k\in\cu A$.
By quantifier elimination, every formula of $T^R$ with parameters in $\cu A\cup C$ is equivalent to a formula of that form.
Therefore we have $\cu A'\equiv_C\cu A$ in $\cu N$.  Add a constant symbol to $\cu N$ for each $a\in A$, $\sa F\in\cu A$, and
$\sa F'\in\cu A'$. Add a set of variables $A'=\{\bo a'\mid \bo a\in A\}$.  Consider the set of conditions
$$ \Gamma=\{d_\BB(\sa F',\l\theta(\bar{\bo a}',\bar{\bo c})\rr)=0\mid \cu N\models d_\BB(\sa F,\l\theta(\bar{\bo a},\bar{\bo c})\rr)=0\}$$
with the set of variables $A'$.   It follows from fullness that every finite
subset of $\Gamma$ is satisfiable by some $A'$ in $\cu N$.  Then by  saturation, $\Gamma$ is satisfied by some $A'$ in $\cu N$.
By quantifier elimination we have $A\equiv_C A'$ in $\cu N$, and by definition, $A'\ind[d\BB]_C B$.
\end{proof}

\begin{lemma}  \label{l-fB-cbased}
$\ind[d\BB] \ \ $ is countably based.
\end{lemma}

\begin{proof}
By Lemma \ref{l-cbased} (3), it is enough to check that, for countable $A$, $B$ and small $C$, we have
$$
A\ind[d\BB]_C \ B \Leftrightarrow (\forall^c C'\subseteq C)(\exists^c D\in [C',C])A\ind[d\BB]_D \ B.
$$
Fix such $A$, $B$, $C$.  The forward direction follows from Corollary \ref{c-preserve} (1).  For the other direction, suppose that $A\nind[d\BB]_C \ B$.
Then $\cu A_C\nind[d]_{\cu C} \ \cu B_C$ over $(\cu E,\mu)$.  Hence for some tuple $\bar A$ in $\cu A_C$ and some continuous formula
$\Phi(\bar X, \cu B_C,\cu C)$, $(\cu E,\mu)\models \Phi(\bar A,\cu B_C,\cu C)=0$ and $\Phi(\bar X,\cu B_C,\cu C)$ divides over $\cu C$.
Take a countable (even finite) $C'\subseteq C$ such that $\Phi(\bar x,\cu B_C,\cu C)=\Phi(\bar x,\cu B_{C'},\cu C')$.
Then for any countable $D\in[C',C]$, $\Phi(\bar x,\cu B_C,\cu C)$ divides over $\cu D$.  Therefore $\cu A_D\nind[d]_{\cu D} \cu B_D$ over $(\cu E,\mu)$,
so $A\nind[d\BB]_D \ B$.
\end{proof}

\begin{lemma}  \label{l-fB-local}
$\ind[d\BB] \ \ $ has small local character.
\end{lemma}

\begin{proof}
By Result \ref{f-probability-algebra}, $\ind[d]$ on $(\cu E,\mu)$ has small local character.
To prove that $\ind[d\BB] \ \ $ has small local character, let $A, B$ be small subsets of $\cu K$.
Then $\cu A, \cu B$ are small subsets of $\cu E$.  Let $C_0$ be a subset of $B$ with $|C_0|\le|A|+\aleph_0$.
We show by induction that there is a countable chain $C_0\subseteq C_1\subseteq\cdots$ of subsets of $B$ such that $\cu A_{\cu C_n}\ind[d]_{\cu C_{n+1}} \cu B$
and $|C_n|\le |A|+\aleph_0$ for each $n\in\BN$.   Suppose we have  sets $C_0\subseteq\cdots\subseteq C_n$ of $B$ such that
$\cu A_{\cu C_m}\ind[d]_{\cu C_{m+1}} \cu B$ and $|C_m|\le|A|+\aleph_0$ whenever $m<n$.
(This holds vacuously when $n=0$).  Then $\cu A_{\cu C_n}\subseteq\cu E$.
By small local character for $\ind[d]$ on $(\cu E,\mu)$,
there is a  set $\cu D\subseteq\cu B$ such that $\cu A_{\cu C_n}\ind[d]_{\cu D} \cu B$ and $|\cu D|\le|A|+\aleph_0$.  There is also a  set $C_{n+1}\in[C_n,B]$
such that $\cu D\subseteq\cu C_{n+1}$ and $|C_{n+1}|\le |A|+\aleph_0$.  By base monotonicity, we have $\cu A_{\cu C_n}\ind[d]_{\cu C_{n+1}} \cu B$.  This completes the induction.

Now let $C=\bigcup_n C_n$.  Then $C\subseteq B$ and $|C|\le|A|+\aleph_0$.  We have $\cu C=\bigcup_n\cu C_n$, $\cu A_{\cu C}=\bigcup_n\cu A_{\cu C_n}$, and $\cu B=\cu B_{\cu C}$.
By monotonicity for $\ind[d]$,  $\cu A_{\cu C_n}\ind[d]_{\cu C_k} \cu B$ whenever $n<k$.  By Lemma \ref{l-d-union},
$\ind[d]$ has the countable union property.  Therefore
$\cu A_{\cu C_n}\ind[d]_{\cu C} \cu B$ for each $n$.  Then by finite character for $\ind[d]$, $\cu A_{\cu C}\ind[d]_{\cu C}\cu B$, so $A\ind[d\BB]_C \ B.$
\end{proof}

\begin{cor}  \label{c-fB-indep}
In $T^R$, $\ind[d\BB] \ \ $ is an independence relation, and $\ind[d]\Rightarrow\ind[d\BB] \ $.
\end{cor}

\begin{proof}  By Lemmas \ref{l-fB-basic}, \ref{l-fB-existence}, and \ref{l-fB-local}, $\ind[d\BB] \ \ $ is an independence relation.
By Remark 5.1 of [Ad2], $\ind[d]$ implies any independence relation in a first order theory, and the proof carries over to continuous theories,
so $\ind[d]\Rightarrow\ind[d\BB] \ $ in $T^R$.
\end{proof}

\begin{prop}  \label{p-ind[fB]-not-antireflexive}
If $T$ has an infinite model, then $\ind[d\BB] \ \ $ is not anti-reflexive.
\end{prop}

\begin{proof}
Since $\cu M$ is large, it has an element $a$ that is not definable over $\emptyset$.  By fullness and saturation, there is an element $\bo a\in\cu K$
such that $\mu(\l\varphi({\bo a})\rr)=1$ whenever $\cu M\models \varphi(a)$, so
$\fo_\BB({\bo a})=\sigma(\emptyset)$.   Hence
${\bo a} \ind[d\BB]_{\emptyset}{\bo a}$.  But $\bo a\notin\dcl^\omega(\emptyset)$.  So ${\bo a}\notin\dcl(\emptyset)=\acl(\emptyset)$ by Results \ref{f-acl=dcl} and\ref{f-dcl3},
so $\ind[d\BB]$ \ \ is not anti-reflexive.
 \end{proof}

\subsection{Algebraic Independence in the Event Sort}

\begin{df}  \label{d-event-indep}
For small $A, B, C\subseteq \cu K$, define
 $$A\ind[a\BB]_C \ \ B\Leftrightarrow\acl_\BB(AC)\cap\acl_\BB(BC)=\acl_\BB(C).$$
\end{df}

\begin{rmk}  \label{r-ind-aB}
$A\ind[a\BB]_C \ \ B$ in $\cu N\Leftrightarrow \cu A_C\ind[a]_{\cu C} \cu B_C$ in $(\cu E,\mu)$.
\end{rmk}

\begin{lemma}  \label{l-ind[fB]-implies-ind[aB]}
$\ind[d\BB] \ \Rightarrow \ind[a\BB] \ .$
\end{lemma}

\begin{proof}  By Result \ref{f-probability-algebra}, the theory of $(\cu E,\mu)$ is stable, so
$\ind[d]\Rightarrow \ind[a]$ in $(\cu E,\mu)$.  It follows at once that $\ind[d\BB] \ \Rightarrow\ind[a\BB] \ $ in $\cu N$.
\end{proof}

\begin{prop} \label{p-event-aB}
The relation $\ind[a\BB] \ \ $ over $\cu N$ satisfies all the axioms for a countable independence relation except base monotonicity.
It also has symmetry, the countable union property, and countably local character.
\end{prop}

\begin{proof}  Symmetry, invariance, monotonicity, normality, countable character, and the countable union property are clear.
The proof for transitivity is the same as in Proposition \ref{l-fB-basic}.
Full existence, extension, and small local character follow from Lemmas \ref{l-fB-existence}, \ref{l-fB-local}, and  \ref{l-ind[fB]-implies-ind[aB]}.
\end{proof}

\section{Pointwise Independence}  \label{s-pointwise}

\subsection{The General Case}

\begin{df}  \label{d-pointwise}
If $\ind[I]$ is a ternary relation over $\cu M$ that has monotonicity, we let $\ind[I \omega]$ \quad
be the unique countably based relation over $\cu N$ such that for all countable $A,B,C$,
$$ A\ind[I \omega]_C \  B\Leftrightarrow A(\omega)\ind[I]_{C(\omega)} B(\omega) \ \as.$$
The unique existence of $\ind[I \omega]$ \quad follows from Lemma \ref{l-cbased} (2).
We say that \emph{$A$ is pointwise $I$-independent from $B$ over $C$} if $A\ind[I \omega]_C \ B$.
\end{df}

We will often use the notation $\l P\rr$ for the set $\{\omega\in\Omega\mid P(\omega)\}$
when $P(\omega)$ is a statement involving elements $\omega$ of $\Omega$.
Since $(\Omega,\cu F,\mu)$ is a complete probability space,  $P(\omega)$ holds  $\as$
if and only if $\mu(\l P \rr)=1$.
For instance, if $\ind[I]$ is a ternary relation over $\cu M$, then for all countable sets $A,B,C\subseteq\cu K$,
$$\l A\ind[I]_C B\rr=\{\omega\in\Omega \, : \, A(\omega)\ind[I]_{C(\omega)} \ B(\omega)\},$$
and
$$ A\ind[I \omega]_C \  B\Leftrightarrow\mu(\l A\ind[I]_C B\rr)=1.$$

\begin{cor} \label{c-pointwise-implies}
If $\ind[I]$ and $\ind[J]$ are  ternary relations over $\cu M$ with monotonicity, and $\ind[I]\Rightarrow\ind[J]$, then
$\ind[I \omega] \ \ \Rightarrow \ind[J \omega] \ \ $.
\end{cor}

\begin{proof}  This follows from Lemma \ref{l-cbased} (1).
\end{proof}

\begin{df}  A ternary relation $\ind[J]$ over $\cu N$ will be called \emph{pointwise anti-reflexive} if $\bo a\ind[J]_C \bo a$ implies $\bo a\in\acl^\omega(C)$.
\end{df}

Note that every anti-reflexive relation over $\cu N$ is pointwise anti-reflexive.

\begin{prop}  \label{p-omega-cbased}
Suppose $\ind[I]$ is a  ternary relation over $\cu M$ that has monotonicity.
\begin{enumerate}
\item  $\ind[I \omega]$ \ \ is countably based and has monotonicity and two-sided countable character.
\item If $\ind[J]=\ind[Ic] \ $, then $\ind[J\omega] \ =\ind[I\omega] \ $.
\item If $\ind[I]$ has invariance,  transitivity, normality, symmetry, or the countable union property,
then  $\ind[I \omega]$ \ \ has the same property.  If $\ind[I]$ has
 normality, symmetry, and base monotonicity, then $\ind[I\omega] \ \ $ has base monotonicity.
\item If $\ind[I]$ is anti-reflexive, then $\ind[I\omega] \ \ $ is pointwise anti-reflexive.
\end{enumerate}
\end{prop}

\begin{proof}  (1) By Definition \ref{d-pointwise}, $\ind[I \omega]$ \ \ is countably based.  It has monotonicity and two-sided countable character by Lemma \ref{l-cbased} (3).

(2)  This follows from the fact that $\ind[I]$ and $\ind[J]$ agree on countable sets.

(3) By (2), Proposition \ref{p-preserve}, and Remark \ref{r-union-cbased}, it suffices to show that if $\ind[I]$ has one of the listed properties for countable sets, then
$\ind[I \omega]$ \ \ has the same property for countable sets.

 We prove the result for transitivity.  The other proofs are similar but easier.
 Suppose $\ind[I]$ has transitivity for countable sets, and assume that $A,B,C,D$
 are countable and $B\ind[I \omega]_C \  A$, $C\ind[I \omega]_D \  A$ and $C\in[D,B]$.
We must prove $B\ind[I \omega]_D \ A$. We have $\mu(\l B\ind[I]_C \ A\rr)=1$, $\mu(\l C\ind[I]_D \ A\rr)=1$, and $\mu(\l C\in [D,B]\, \rr)=1$.  Since
$\ind[I]$ has transitivity for countable sets, $\mu(\l B\ind[I]_D \ A\rr)=1$.  This shows that $B\ind[I \omega]_D  \ A$.

(4)  Suppose $\ind[I]$ is anti-reflexive and $\bo a\ind[I\omega]_C \ \ \bo a$.  Then $\bo a\ind[I\omega]_D \ \ \bo a$ for some countable $D\subseteq C$,
so $a(\omega)\ind[I]_{D(\omega)} \ \ a(\omega) \as$, and hence $\bo a\in\acl^\omega(D)\subseteq\acl^\omega(C)$.
\end{proof}

To sum up, we have:

\begin{cor}  \label{c-summary}  If $\ind[I]$ satisfies the basic axioms for an independence relation over $\cu M$ ,
then $\ind[I \omega]$ \ \ over $\cu N$ also satisfies these axioms.
\end{cor}

\begin{prop}  \label{p-pointwise-finitechar}
Suppose $\ind[I]$ is a  ternary relation over $\cu M$ that is countably based and has finite character and the countable union property.
Then  $\ind[I \omega]$ \ \ has finite character.
\end{prop}

\begin{proof}  Suppose $A'\ind[I \omega]_C \ B$ for all finite $A'\subseteq A$.  Let $A_0\subseteq A, B_0\subseteq B, C_0\subseteq C$
be countable.  We must find a countable $D\in[C_0,C]$ such that $A_0\ind[I\omega]_{D} \ B_0$.  We may write $A_0=\bigcup_n E_n$ where
$E_0\subseteq E_1\subseteq\cdots$ and each $E_n$ is finite.  Then $E_n\ind[I \omega]_C B$ for each $n$.

Since $\ind[I \omega] \  $ is countably
based, there are countable sets $\<D_n\>_{n\in\BN}$ such that for each $n$, $D_n\in[C_0,C]$, $D_n\subseteq D_{n+1}$, and $E_n\ind[I \omega]_{D_n} \ B_0$,
and hence $E_n\ind[I]_{D_n} B_0 \ \as.$    Let $D=\bigcup_n D_n$.
Since $\ind[I]$ has the countable union property, for each $n$ we have $ E_n\ind[I]_{D} B_0 \ \as$.
Since $\ind[I]$ has finite character, $ A_0\ind[I]_{D} B_0 \ \as$., so $A_0\ind[I\omega]_D \ B_0$.
\end{proof}

\subsection{Measurable Ternary Relations}

\begin{df}  \label{d-measurable}
We say that a ternary relation $\ind[I]$ over the big model $\cu M$ of $T$ is \emph{measurable} if
 $\l A\ind[I]_C B\rr\in\cu F$ for all countable $A,B,C\subseteq\cu K$.
\end{df}

We will sometimes use measurability without explicit mention in the following way:  if $\ind[I]$ is measurable, $A,B,C$ are countable, and $A \nind[I]_C B$, then $\mu(\l A\ind[I]_C B\rr)=r$ for some $r<1$.

In Lemma \ref{l-conjunctive-measurable} below, we will give a  useful sufficient condition for measurability.
 $L_{\omega_1\omega_1}$ is the infinitary logic that contains first order logic and is closed under finite or countable conjunctions and disjunctions, negations,
and finite or countable existential and universal quantifiers.   An $L_{\omega_1\omega_1}$ formula is said to be \emph{conjunctive}  if it is equivalent to a formula that is built from first order
formulas using only finite or countable conjunctions, finite or countable quantifiers, and finite  disjunctions.  By a \emph{Borel-conjunctive} formula we mean an
$L_{\omega_1\omega_1}$ formula that is built from conjunctive formulas using only negations and finite and countable conjunctions and disjunctions.

The following result is a consequence of Theorem 2.3 in [Ke2]. It depends on the fact that $\cu M$ is $\aleph_1$-saturated.

\begin{result}  \label{f-conjunctive}
For every countable set $X$ of variables and conjunctive $L_{\omega_1\omega_1}$-formula $\theta(X)$
there is a countable set $\cu A(\theta)$ of first order formulas (the set of finite approximations of $\theta$) such that
$$\cu M\models (\forall X)\left[\theta(X)\leftrightarrow\bigwedge\left\{\psi(X)\colon\psi\in\cu A(\theta)\right\}\right].$$
\end{result}

We say that $\ind[I]$ is \emph{definable by} an $L_{\omega_1\omega_1}$  formula $\varphi(X,Y,Z)$ in $\cu M$ with countable sets of variables $X,Y,Z$ if for all countable sets
$A, B, C$ indexed by $(X,Y,Z)$ in $\cu M$, we have
$$A\ind[I]_C B\Leftrightarrow\cu M\models\varphi(A,B,C).$$

\begin{lemma}  \label{l-conjunctive-measurable}
Suppose $\ind[I]$ is definable by a Borel-conjunctive formula $\varphi$ in $\cu M$.  Then $\ind[I]$ is measurable, and for all countable sets $A, B, C\subseteq\cu K$,
the reduction of $\l A\ind[I]_C B\rr$ belongs to $\dcl_\BB(ABC)$.
\end{lemma}

\begin{proof}  By Result \ref{f-conjunctive}, for every conjunctive formula $\theta(X,Y,Z)$ and all countable
sets $A,B,C\subseteq\cu K$ we have
$$\l\theta(A,B,C)\rr=\bigcap\{\l\psi(A,B,C)\rr\colon\psi\in\cu A(\theta)\}.$$
By definition, $\l\psi(A,B,C)\rr\in\cu F$ for every first order formula $\psi$.  Since $\cu F$ is a $\sigma$-algebra,
it follows that $\l\theta[A,B,C)\rr\in\cu F$ for every conjunctive formula $\theta$.  Since $\varphi$ is Borel-conjunctive, we have
$$\l A\ind[I]_C B\rr=\l\varphi(A,B,C)\rr\in\cu F.$$
Because the reduction of $\l\psi(A,B,C)\rr$ belongs to $\fo_\BB(ABC)$ for each first order formula $\psi$, and $\mu$ is $\sigma$-additive,
we see from Result \ref{f-separable} that the reduction of $\l\varphi(A,B,C)\rr$ belongs to $\dcl_\BB(ABC)$.
\end{proof}

\begin{thm}  \label{t-measurable}
For every complete theory $T$, each of the relations
$$\ind[a], \qquad \ind[M] \ , \qquad \ind[b], \qquad \ind[e], \qquad \ind[d]$$
is definable by a Borel-conjunctive formula and hence is measurable over $\cu M.$
\end{thm}

\begin{proof}
$\ind[a]$ is defined by the Borel-conjunctive formula
$$\neg\bigvee_i\bigvee_j(\exists u)[\varphi_i(u,X,Z)\wedge \psi_j(u,Y,Z) \wedge \bigwedge_k \neg\chi_k(u,Z)],$$
where $\varphi_i(u,X,Z), \psi_j(u,Y,Z), \chi_k(u,Z)$ list all algebraical formulas with the indicated variables.

$\ind[M] \ $:  Let $\vec v_n$ denote the $n$-tuple of variables $\<v_1,\ldots,v_n\>$.
Let $\eta_i(u,Y,Z)$ enumerate all algebraical formulas over $Y,Z$.  Let $\varphi_i^n(u,\vec v_n,X,Y,Z)$ enumerate all formulas
of the following form:
$$\bigwedge_{j=1}^n \eta_{i_j}(v_j,Y,Z)\wedge \psi(u,\vec v_n,X,Z)\wedge \chi(u,\vec v_n,Y,Z),$$
where $\psi$ and $\chi$ are also algebraical formulas; here $n$ is allowed to vary.  Let $\theta_k^n(u,\vec v_n,Z)$ enumerate all algebraical
formulas over $\vec v_n,Z$.  Then $\ind[M] \ $ is defined by the Borel-conjunctive formula
$$
\neg\bigvee_n\bigvee_i(\exists u)(\exists \vec v_n)\left[\varphi_i^n(u,\vec v_n,X,Y,Z)\wedge \bigwedge_k\neg \theta_k^n(u,\vec v_n,Z))\right].$$

$\ind[b]$ is defined by the Borel-conjunctive formula
$$\bigwedge_{i} \bigwedge_{\bar {x}\in X^{<\BN}}\bigwedge_{\bar y\in Y^{<\BN}}\bigwedge_{\bar{z}\in Z^{<\BN}}\bigvee_{\bar {u}\in Z^{<\BN}}
[\varphi_i(\bar{x},\bar{y},\bar{z})\Rightarrow\varphi_i(\bar{x},\bar{u},\bar{z})]$$
where $\varphi_i(\bar x,\bar y,\bar z)$ lists all formulas in $[L]$.

$\ind[e]$ is defined by the Borel-conjunctive formula
$$\bigwedge_{u\in X}\bigwedge_{i} \bigwedge_{\bar {x}\in X^{<\BN}}\bigwedge_{\bar{y}\in Y^{<\BN}}\bigwedge_{\bar{z}\in Z^{<\BN}}
\bigvee_{j}\bigvee_{\bar{v}\in Z^{<\BN}}
[\varphi_i(u,\bar{x},\bar{y},\bar{z})\Rightarrow\psi_j(u,\bar{x},\bar{v})]$$
where $\varphi_i(u,\bar x,\bar y,\bar z)$ and $\psi_j(u,\bar x,\bar z)$ list all algebraical formulas with the indicated variables.

$\ind[d]$: Let $\varphi(\vec x,\vec y,Z)$ be a first order formula, where
$\vec x,\vec y$ are tuples of variables, and $Z$ is a countable set of variables.
For each tuple $\vec b$ and countable set $C$ in $\cu M$, $\varphi(\vec x, \vec b, C)$ divides over $C$
if and only if  $(\cu M,\vec b, C)$ satisfies the following Borel-conjunctive formula $\div_{\varphi}(\vec y,Z)$:
$$ \bigvee_k(\exists \vec u^{ \, 0}\exists\vec u^{ \, 1}\ldots)
\left[IND \wedge\vec u^{\, 0}\equiv_Z \vec y
\wedge \bigwedge_{I\subset \BN, |I|=k}\neg(\exists\vec x)\bigwedge_{i\in I} \varphi(\vec x,\vec u^{\, i},Z) \right],$$
where $|\vec u^{ \, j}|=|\vec y|$ for each $j$, and $IND$ is the conjunctive formula that says that the sequence $(\vec u^{ \, 0},\vec u^{ \, 1},\ldots)$ is indiscernible over $Z$.
Therefore $\ind[d]$ is defined by the Borel-conjunctive formula
$$ \neg\bigvee_{\vec x\in X^{<\BN}}\bigvee_{\vec y\in Y^{<\BN}}\bigvee_\varphi (\varphi(\vec x,\vec y,Z)\wedge \div_\varphi(\vec y,Z)).$$
\end{proof}

Note that if $T$ has the exchange property, then by Theorem \ref{t-indep-FO} (1), $\thind=\ind[e]$ in models of $T$, so by Theorem \ref{t-measurable},
$\thind$ is measurable over $\cu M$.

\subsection{Pointwise Dividing and Stability}  \label{s-dividing}

We will establish a relationship between dividing and pointwise dividing  in the randomization of a stable theory.
Consider the relation $\ind[d\wedge] \ =\ind[d\omega] \ \ \wedge \ind[dB] \ \ $.  Proposition \ref{l-simple-omega} will show that  when $T$ is stable,
 $\ind[d\wedge] \ \ $  is an independence relation in $T^R$.
Moreover,  $\ind[d\wedge] \ =\ind[d]$ in $T^R$ if and only if $T$ has $\acl=\dcl$.

\begin{lemma}  \label{l-ind[d]-pointwise}
Suppose $A,B,C\subseteq\cu K$ are countable and $A\ind[d]_C B$ in $\cu N$.  Then $A\ind[d\omega]_C \ B$.
\end{lemma}

\begin{proof}
We will use the notation from the proof of Theorem \ref{t-measurable}.  Suppose that $A\nind[d\omega]_C \ B$.
Then the set $\sa E=\l A\ind[d]_C B\rr$ belongs to $\cu F$,  and $\mu(\sa E)<1$.  Hence there are a first-order formula $\varphi(\vec x,\vec y,Z)$,
and tuples $\vec {\bo a}\in A^{<\BN}, \vec {\bo b}\in B^{<\BN}$ such that
$$\mu(\l\varphi(\vec a,\vec b,C)\wedge\div_\varphi(\vec b,C)\rr)>0.$$
For each $k\in\BN$, let $\div^k_\varphi(\vec y,Z)$ be the part of $\div_\varphi(\vec y,Z)$ after the initial $\bigvee_k$.
Then $\div^k_\varphi(\vec y,Z)$ is a conjunctive formula,
$\l \div^k_\varphi(\vec b,C)\rr\in\cu F$, and  $\l \div_\varphi(\vec b,C)\rr=\bigcup_k\l \div^k_\varphi(\vec b,C)\rr$,
so we may find a $k\in\BN$ such that
$$r:=\mu(\l\varphi(\vec a,\vec b,C)\wedge \div^k_\varphi(\vec b,C)\rr)>0.$$
By Result \ref{f-conjunctive}, there is a countable set $\{\theta_m(\vec y,Z)\colon m\in\BN\}$  of first order formulas
closed under finite conjunction such that
$$ \cu M\models(\forall\vec y,Z)\left[\div^k_\varphi(\vec y,Z)\leftrightarrow\bigwedge_m \theta_m(\vec y,Z)\right].$$
Therefore
$$ \l\div^k_\varphi(\vec b, C)\rr=\bigcap_m \l\theta_m(\vec b,C)\rr,$$
so there exist $m(k)\in\BN$ such that
$$\mu(\l\theta_{m(k)}(\vec b,C)\wedge\neg\div^k_\varphi(\vec b,C)\rr)\le r/2.$$
Now let $\Phi(\vec x,\vec{\bo b},C)$ be the continuous formula
$$r\dotminus\mu(\l\varphi(\vec x,\vec b,C)\wedge\theta_{m(k)}(\vec b,C)\rr).$$
We have
$$\mu(\l\varphi(\vec a,\vec b,C)\wedge\theta_{m(k)}(\vec b,C)\rr)\ge r,$$
so $\Phi(\vec {\bo a},\vec {\bo b},C)=0$.

\emph{Claim.} $\Phi(\vec x,\vec {\bo b},C)$ divides over $C$ in $\cu N$.

\emph{Proof of Claim:}  Using Ramsey's theorem and the fact that $\cu M$ is saturated, one can show that for each $\omega\in \l\div^k_\varphi(\vec b,C)\rr$,
there is a  sequence $\<\vec {b'}_i\>_{i\in\BN}$ in $\cu M$ such that:
\begin{itemize}
\item[(a)] $\<\vec {b'}_i\>_{i\in\BN}$ is $C(\omega)$-indiscernible,
\item[(b)] $\vec{b'}_0\equiv_{C(\omega)} \vec b(\omega)$, and
\item[(c)] $\cu M\models\neg(\exists\vec x)\bigwedge_{i<k}\varphi(\vec x,\vec{b'}_i,C(\omega))$.
\end{itemize}
By $\omega_1$-saturation for $\cu N$ and fullness, there is a sequence $\<\vec{\bo b}''_i\>_{i\in\BN}$ in $\cu K$ such that for all $\omega\in \l\div^k_\varphi(\vec b,C)\rr$,
conditions (a)--(c) above hold when $\vec b'_i=\vec b''_i(\omega)$.
By Result \ref{f-glue}, for each $i\in\BN$ there is a $\vec{\bo b}_i$ that agrees with $\vec{\bo b}'_i$ on $\l\div^k_\varphi(\vec b,C)\rr$, and agrees with $\vec{\bo b}$ elsewhere.  Then
$\<\vec{\bo b}_i\>_{i\in\BN}$ is $C$-indiscernible, and $\vec{\bo b_0}\equiv_C\vec{\bo b}$ in $\cu N$.
Consider a tuple $\vec{\bo d}\in\cu K^{|\vec x|}$, and for each $i\in\BN$ let
$$\sa D_i=\l\varphi(\vec d,\vec b_i,C)\wedge\div^k_\varphi(\vec b,C)\rr.$$
By conditions (a) and (c) above,
for all distinct elements $i_1,\ldots,i_k$ of $\BN$, we have $\mu(\bigcap_{j=1}^k \sa D_{i_j})=0$.
Therefore, by Lemma 7.5 of [EG], there exists $i\in\BN$ such that
$\mu(\sa D_i)<r/2$.  By (a) and (b) above,
$\l\div^k_\varphi(\vec b,C)\rr=\l\div^k_\varphi(\vec b_i,C)\rr$, so
$$\sa D_i=\l\varphi(\vec d,\vec b_i,C)\wedge\div^k_\varphi(\vec b_i,C)\rr.$$
Hence
$$\mu(\l\varphi(\vec d,\vec b_i,C)\wedge\theta_{m(k)}(\vec b_i,C)\rr)\le$$
$$\mu(\sa D_i)+\mu(\l(\theta_{m(k)}(\vec b_i,C)\wedge\neg\div^k_\varphi(\vec b_i,C)\rr)
<r/2+r/2=r,$$
and thus $\Phi(\vec{\bo d},\vec{\bo b}_i,C)>0.$
This proves the Claim.

It follows that $\vec{\bo a}\nind[d]_C\vec{\bo b}$ in $\cu N$, so by monotonicity, $A\nind[d]_C B$ in $\cu N$.
\end{proof}

\begin{result}  \label{f-stable-TR}  ([BK], Theorem 5.13.)
If $T$ is stable, then $T^R$ is stable.
\end{result}

\begin{prop}  \label{l-simple-omega}  Suppose that $T$ is stable, and  let $\ind[d\wedge] \ =\ind[d\omega] \ \ \wedge\ind[d\BB] \ $ on $\cu N$.  Then on $\cu N$ we have:
\begin{enumerate}
\item $\ind[d\wedge] \  \ $ is countably based.
\item $\ind[d]\Rightarrow \ind[d\wedge] \  \ $.
\item $\ind[d\wedge] \  \  $ and $\ind[d\omega] \ \ $ are independence relations with small local character.
\item $\ind[d\wedge] \  \  = \ind[d]$ in $T^R$ if and only if $T$ has $\acl=\dcl$.
\end{enumerate}
\end{prop}

\begin{proof}   By  Proposition \ref{p-d-cbased},  Lemma \ref{l-d-union}, and Results \ref{f-stable-TR} and \ref{f-stable-indep},
over both $\cu M$ and $\cu N$, $\ind[d]$ is a strict countably
based independence relation that has the countable union property, stationarity, and small local character.

(1): $\ind[d\omega] \ \ $ and $\ind[d\BB] \ \ $ are each countably based and have the countable union property by Proposition
\ref{p-d-cbased} and Lemmas \ref{l-fB-basic}, \ref{l-fB-cbased}, and \ref{l-d-union}.  Then $\ind[d\wedge] \  \ $ is countably based by
Proposition \ref{p-union-pair}.

(2): By  Lemmas  \ref{l-cbased} (1) and \ref{l-ind[d]-pointwise} and the fact that $\ind[d]$ on $\cu N$ is countably based,
we have $\ind[d]\Rightarrow\ind[d\omega] \ \ $ on $\cu N$.
By Corollary \ref{c-fB-indep}, $\ind[d]\Rightarrow\ind[d\BB] \ \ $.

(3):  By (2) above, Corollary \ref{c-summary},  Proposition \ref{p-preserve-finitechar}, and Remark \ref{r-weaker}, $\ind[d\omega] \ \ $
is an independence relation with small local character.  By Lemma \ref{l-fB-local}, Corollary \ref{c-fB-indep}, and  Remark \ref{r-weaker},
$\ind[d\BB] \ \ $ is also an independence relation with small local character.
It then follows easily that $\ind[d\wedge] \  \ $ satisfies the basic axioms and finite character.  By (2) above and Remark \ref{r-weaker} again,
$\ind[d\wedge] \  \ $ also satisfies extension and small local character.

(4):  First assume that $T$ does not have $\acl=\dcl$.  Take a finite set $C$ and an element $a$ in $\cu M$ such that $a\in\acl(C)\setminus\dcl(C)$.
Let $x\mapsto\ti x$ be the mapping introduced in Lemma \ref{l-big-embedding}.
Then in $\cu N$ we have $\ti a\ind[d\omega]_{\ti C} \ \ti a \wedge \ti a\ind[d\BB]_{\ti C} \ \ti a$ but $\ti a\notin\dcl(\ti C)=\acl(\ti C)$,
Thus $\ti a\ind[d\wedge]_{\ti C} \ti a$ but $\ti a\nind[d]_{\ti C} \ti a$, and $\ind[d\wedge] \  \ \ne\ind[d]$ in $T^R$.

Now assume that $T$ has $\acl=\dcl$.  We show that $\ind[d\wedge] \  \ $ satisfies stationarity and is anti-reflexive.
It will then follow from Result \ref{f-stable-indep} that $\ind[d\wedge] \  \  = \ind[d]$ in $T^R$.

Let $\bo a\ind[d\wedge]_C \bo a$.  Since $\ind[d]\Rightarrow\ind[a]$ in $\cu M$, we have
$\ind[d\omega] \ \ \Rightarrow \ind[a\omega] \ \ $,
so $\bo a\ind[a\omega]_C \ \bo a$ and $\bo a\in\acl^\omega(C)=\dcl^\omega(C)$.  Also, $\ind[d\BB] \ \Rightarrow \ind[a\BB] \ $, and thus
$\bo a\ind[a\BB]_C \ \bo a$
and $\fo_\BB(\bo a C)\subseteq \acl_\BB(\bo a C)\subseteq\acl_\BB(C)$.  By Result \ref{f-acl=dcl}, $\acl_\BB(C)=\dcl_\BB(C)$.  Therefore by Result \ref{f-dcl3},
$\bo a\in\dcl(C)=\acl(C)$, so $\ind[d\wedge] \  \ $ is anti-reflexive.

To prove that $\ind[d\wedge] \  \ $ has stationarity,
suppose that in $\cu N$, $C$ is algebraically closed, $\bar a\equiv_C  \bar g$, $\bar a\ind[d\wedge]_C \ BC$, and $\bar g\ind[d\wedge]_C \ BC$.
We must prove that $\bar a\equiv_{CB}\bar g$.  By using characteristic functions of events, we may assume that
$\bar a, \bar g, B$, and $C$ have sort $\BK$.  It is enough to show that for each first order formula $\varphi(\bar x,\bar y,\bar z)$
and all tuples $\bar b$ in $B$ and $\bar c$ in $C$ we have  $\l \varphi(\bar a,\bar b,\bar c)\rr\doteq \l \varphi(\bar g,\bar b,\bar c)\rr.$

We have $\bar a\ind[d\omega]_C \ \bar b$, and $\bar g\ind[d\omega]_C \ \bar b$.
 It follows by induction that there is an increasing chain
$\bar c\subseteq C_0\subseteq C_1\subseteq\cdots$ of countable subsets of $C$ such that for each $n$,
$C_n=\fo(C_n)$, $\bar a\ind[d\omega]_{C_n} \ \bar b$
when $n$ is even, and $\bar g\ind[d\omega]_{C_n} \ \bar b$ when $n$ is odd.  Let $D=\bigcup_n C_n$.  Then $\bar c\subseteq D\subseteq C$,
$D$ is countable, and $D=\fo(D)$.  By (1) above, $\ind[d \omega] \ \ $ has the countable union property,
so $\bar a\ind[d\omega]_{D} \  \bar b$ and $\bar g\ind[d\omega]_{D} \ \bar b$.
By symmetry and normality, $\bar a\ind[d]_{D} \ \bar b D$ and $\bar g\ind[d]_{D} \ \bar b D$.
Since $D$ is countable, we have
$$\mu(\l \bar a\ind[d]_{D} \ \bar b D\rr)=\mu(\l \bar g\ind[d]_{D} \ \bar b D\rr)=\mu(\l \bar a\equiv_D \bar g\rr)=1.$$
Because  $T$ has $\acl=\dcl$ and $D=\fo(D)$, we have $\mu(\l D=\acl(D)\rr)=1$.  Then since $\ind[d]$ in $T$ has stationarity,
$\mu(\l \bar a\equiv_{\bar b D} \bar g \rr)=1.$  Therefore $\l \varphi(\bar a,\bar b,\bar c)\rr\doteq \l \varphi(\bar g,\bar b,\bar c)\rr$,
as required.
\end{proof}

\subsection{Pointwise Blanketing Independence} \label{s-blanketing}

In $\cu N$, we say that $C$ \emph{blankets} $A$ in $B$ if $A\ind[b\omega]_C \ B.$

\begin{cor}  The relation $\ind[b\omega] \ \ $ over $\cu N$ has invariance, monotonicity, and normality.
\end{cor}

\begin{proof}
This is an immediate consequence of Lemma \ref{l-ind[b]-easy}.
\end{proof}

\begin{rmk}  The following facts are left as exercises for the reader.  They will not be used in this paper.
\begin{itemize}
\item[(1)] The relation $\ind[b] \ $ has the countable union property,
and is  countably based.
\item[(2)] The relation $\ind[b\omega] \ \ \ $ has  finite character and the countable union property,  (Hint: For finite character,
use  Proposition \ref{p-pointwise-finitechar}.)
\end{itemize}
\end{rmk}

\begin{lemma}  \label{l-blanket-local-char}
The relation $\ind[b\omega] \  $ on $\cu N$  has small local character.
\end{lemma}

\begin{proof}
Let $A, B$ be subsets of $\cu K$, with $A$ small. Let $C_0\subseteq B$ be such that  $|C_0|\le|A|+\aleph_0$, For each
$\varphi(\bar x,\bar y,\bar z)\in[L]$, $\bar{\bo a}\in A^{|\bar x|}$, and $\bar{\bo c}\in B^{|\bar z|}$,  we may
take a countable set $D=D(\varphi,\bar{\bo a},\bar{\bo c})\subseteq B$ such that
$$\mu\left(\bigcup_{\bar{\bo d}\in D^{|\bar y|}}\l\varphi(\bar{\bo a},\bar{\bo d},\bar{\bo c})\rr\right)\ge
\mu\left(\bigcup_{\bar{\bo b}\in E^{|\bar y|}}\l\varphi(\bar{\bo a},\bar{\bo b},\bar{\bo c})\rr\right)$$
for every countable set $E\subseteq B.$  Without loss of generality, $\bar{\bo c}$ is contained in $D(\varphi,\bar{\bo a},\bar{\bo c})$.  Then
$$\l\varphi(\bar{\bo a},\bar{\bo b},\bar{\bo c})\rr\sqsubseteq\bigcup_{\bar{\bo d}\in D^{|\bar y|}}\l\varphi(\bar{\bo a},\bar{\bo d},\bar{\bo c})\rr$$
for every $\bar{\bo b}\in B^{|\bar y|}$.  We start with the given set $C_0$, and for each $n\in\BN$, let
$$C_{n+1}=C_n\cup\bigcup\{D(\varphi,\bar{\bo a},\bar{\bo c})\mid \varphi(\bar x,\bar y,\bar z)\in[L],
\bar{\bo a}\in A^{|\bar x|}, \bar{\bo c}\in C_n^{|\bar z|}\}.$$
Since the signature $L$ is countable, we see by induction that for each $n$, $|C_n|\le |A|+\aleph_0$ and $C_n\subseteq B$.
Let $C=\bigcup_n C_n$.  Then $|C_n|\le |A|+\aleph_0$ and $C\subseteq B$.
To show that $A\ind[b\omega]_C \ B$,  let $A_0\subseteq A, B_0\subseteq B, C'\subseteq C$ be countable.   Put
$$C''=\bigcup\{D(\varphi,\bar{\bo a},\bar{\bo c})\mid \varphi(\bar x,\bar y,\bar z) \in[L],
\bar{\bo a}\in A_0^{|\bar x|}, \bar{\bo c}\in C'^{|\bar z|}\}.$$
Then $C''$ is countable, $C''\in[C',C]$, and for each $\varphi(\bar x,\bar y,\bar z) \in[L], \bar{\bo a}\in A_0^{|\bar x|}, \bar{\bo c}\in C''^{|\bar z|},$
and $\bar{\bo d}\in B_0^{|\bar y|}$,
$$\l\varphi(\bar{\bo a},\bar{\bo b},\bar{\bo c})\rr\sqsubseteq\bigcup_{\bar{\bo d}\in C''^{|\bar y|}}\l\varphi(\bar{\bo a},\bar{\bo d},\bar{\bo c})\rr.$$
This shows that $A_0\ind[b\omega]_{C''} \ B_0$.  Since $\ind[b\omega]\ \ $ is countably based, it follows that $A\ind[b\omega]_{C} \ B$.
\end{proof}

\subsection{Pointwise Thorn Independence}  \label{s-pointwise-exchange}

In this section, we will prove that if $T$ has the exchange property, then $\ind[\th\omega] \ \ $ is an independence relation.

\begin{lemma}  \label{l-ind[eomega]-easy}  If $T$ has the exchange property, then
$\ind[\th\omega] \ $ satisfies the basic independence axioms and finite character, is pointwise anti-reflexive, has the countable union property,
and has small local character.
\end{lemma}

\begin{proof} The basic axioms, the countable union property, pointwise anti-reflexiveness, and finite character follow from Propositions \ref{p-omega-cbased} and \ref{p-pointwise-finitechar}. By Lemma \ref{l-blanket-local-char},  $\ind[b\omega] \ \ $ has small local character.
By Theorem \ref{t-indep-FO} we have $\ind[b]\Rightarrow\ind[\th]$.  By Corollary \ref{c-pointwise-implies},
$\ind[b\omega] \ \Rightarrow \ind[\th\omega] \ .$  Therefore $\ind[\th\omega] \ \ $ has small local character.
\end{proof}


Our next task is to show that $\ind[\th\omega] \ \ $  satisfies full existence.
We first introduce some notation.  $A, B, C, \cu{Z}$ will always denote small sets, with $A, B, C\subseteq\cu K$ and $ \cu{Z}\subseteq\cu E$.
Given an enumerated subset $A=\{\bo a_\alpha\mid \alpha<\kappa\}$ of $\cu K$,
for each $\alpha<\kappa$ we let $A_\alpha=\{{\bo a}_\beta\mid \beta<\alpha\}$, and for each $I\subseteq\kappa$
we let $A_I=\{{\bo a}_\beta\mid\beta\in I\}.$
Given an automorphism $\bo u\mapsto\bo u'$ of $\cu N$,  we let $A'=\{\bo a'\mid \bo a\in A\}$.

\begin{thm}   \label{l-full-existence} Suppose $T$ has the exchange property.  For any $A, B, C, \cu{Z}$,
there exists $A'\equiv_{C\cu{Z}} A$ such that $A'\ind[\th\omega]_C \ B$.  Hence $\ind[\th\omega] \ \ $  satisfies full existence in $\cu N$.
\end{thm}

\emph{For the remainder of this section, we assume that $T$ has the exchange property}.
\medskip

Before proving Theorem \ref{l-full-existence}, we will prove two special cases that will be used in the main proof.
We will frequently use, without explicit mention, the fact that $\thind=\ind[e]$ in $\cu M$ (Theorem \ref{t-indep-FO} (iii)).

\begin{lemma} \label{l-single}   For any $B,C, \cu{Z}$ and each $\bo a\in\cu K$, there exists $\bo a'\equiv_{C\cu{Z}}\bo a$ such that $\bo a'\ind[\th\omega]_C \ B$.
\end{lemma}

\begin{proof} Let $B, C$ be small.  By Lemma \ref{l-blanket-local-char}, there is a countable set $C_0\subseteq C$ that blankets $\bo a$ in $C$.
Let $\sa E=\Omega\setminus\l\bo a\in\acl(C_0)\rr.$
Let $\Gamma(\bo u)$ and $\Delta(\bo u)$ be the sets of $L^R$-formulas
$$\Gamma(\bo u)=\{\l\theta(\bo a,\bar{\bo c})\rr\sqsubseteq\l\theta(\bo u,\bar{\bo c})\rr\mid  \theta\in[L],\bar{\bo c}\in C^{<\BN}\},$$
$$\Delta(\bo u)=\{\sa E\sqsubseteq\l\neg\psi(\bo u,\bar{\bo b})\rr\mid \psi \mbox{ algebraical, } \bar{\bo b}\in (BC)^{<\BN}\}.$$
Let $\Gamma_0(\bo u,\bar{\bo c}), \Delta_0(\bo u,\bar{\bo c},\bar{\bo b})$ be finite subsets of $\Gamma(\bo u)$ and $\Delta(\bo u)$ respectively.
\medskip

\textbf{Claim 1.} At almost all  $\omega\in\sa E$,
$\Gamma_0(u,\bar{\bo c}(\omega))$ is satisfied by infinitely many elements.

\textit{Proof of Claim 1.}
For each $n\in\BN$, the formula
$$\psi_n(u,\bar v)=\bigwedge \Gamma_0(u,\bar v)\wedge(\exists^{\le n} w)\bigwedge\Gamma_0(w,\bar v)$$
is algebraical.  Since $C_0$ blankets $\bo a$ in $C$,
$$\l \psi_n(\bo a,\bar{\bo c})\rr\sqsubseteq \bigcup_{\bar{\bo d}\in C_0^{<\BN}}\l\psi_n(\bo a,\bar{\bo d})\rr\sqsubseteq\l\bo a\in\acl(C_0)\rr\doteq(\Omega\setminus \sa E).$$
It is clear that $\l\bigwedge\Gamma_0(\bar{\bo a,\bar{\bo c}})\rr=\top$.  Therefore for each $n\in\BN$ we have
$$\sa E\sqsubseteq \l(\exists^{> n} w)\bigwedge\Gamma_0(w,\bar{\bo c})\rr.$$
This proves Claim 1.

It is clear that at all  $\omega\in\Omega$,
$\Delta_0(u,\bar c,\bar b))$ is satisfied by all but finitely many elements.  So
$\Gamma_0(u,\bar c)\cup\Delta_0(u,\bar c,\bar b)$ is satisfiable at almost all $\omega\in\sa E$.
At all $\omega\notin\sa E$, $\Gamma_0(\bo a)\cup\Delta_0(\bo a)$ holds.
Since $\cu N$ has perfect witnesses, $\Gamma_0(\bo u)\cup\Delta_0(\bo u)$ is satisfiable in $\cu N$.
By the saturation of $\cu N$, $\Gamma(\bo u)\cup\Delta(\bo u)$ is satisfiable by some element $\bo a'\in\cu K$.
By $\Gamma(\bo a')$ we have $\l\theta(\bo a,\bar{\bo c})\rr\doteq \l\theta(\bo a',\bar{\bo c})\rr$ for each  $\theta\in[L]$ and $\bar{\bo c}\in C^{<\BN}.$  By quantifier elimination in $T^R$ (Result \ref{f-qe}), it follows that $\bo a'\equiv_{C\cu{Z}}\bo a$ for each small $\cu Z$.

Using $\Delta(\bo a')$, we see that for any countable $B_0\subseteq B$ and $C_1\subseteq C$,
$$\mu(\l\bo a'\in\acl(B_0C_0C_1)\Rightarrow \bo a\in\acl(C_0) \rr)=1.$$
Therefore
$\mu(\l\bo a'\ind[\th]_{C_0 C_1} \ B_0\rr)=1,$ so $\bo a'\ind[\th\omega]_C \ B.$
\end{proof}

\begin{lemma} \label{l-wlog}
Given any $B,C,\cu{Z}$, suppose that:
\begin{enumerate}
\item $A=\{\bo a_\alpha\mid \alpha<\kappa\}$,
\item
$(\forall \alpha<\kappa)(\forall^c G\subseteq A_\alpha C)\l \bo a_\alpha \in\acl(G)\rr\sqsubseteq\l \bo a_\alpha = {\bo a}_0\rr,$
\end{enumerate}
Then there exists $A'\equiv_{C\cu{Z}} A$ such that $A'\ind[\th\omega]_C \ B$.
\end{lemma}

\begin{proof}
By Results \ref{f-acl=dcl} and \ref{f-separable},
$\acl_\BB(ABC)=\sigma(\fo_\BB(ABC))$, so $\acl_\BB(ABC)$ is small. Therefore
we may also assume without loss of generality that $\cu{Z}\supseteq\acl_\BB(ABC)$.
We now define a sequence $\{\bo{a}'_\alpha\mid \alpha<\kappa\}$ by induction on ordinals
so that for all $\beta<\kappa$,
\begin{itemize}
\item[(3)] $A'_\beta\equiv_{C\cu{Z}} A_\beta, \quad A'_\beta\ind[\th\omega]_C \ B.$
\end{itemize}
(3) is trivial when $\beta=0$.
By Lemma \ref{l-single}, we may take $\bo a'_0$ so that (3) holds for $\beta=1$.
By finite character, if $\alpha$ is a limit ordinal and
(3) holds for all $\beta<\alpha$, then (3) holds for $\beta=\alpha$.
Suppose that $\alpha>0$, and we have defined $A'_\alpha=\{\bo{a}'_\beta\mid \beta<\alpha\}$
so that (3) holds when $\beta=\alpha$.
We will find $\bo a'_\alpha$ such that (3) holds when $\beta=\alpha+1$.
Since $A'_\alpha\equiv_{C\cu{Z}} A_\alpha$, there is  an automorphism $\tau_\alpha$ of $\mathcal N$
that is fixed on $C\cu{Z}$ and maps $\bo a_\beta$ to $\bo a'_\beta$ for each $\beta<\alpha$.
\medskip

\textbf{Claim 2.}  There exists an element ${\bo u}\in\cu K$ such that:
\begin{itemize}
\item[(a)] For every formula $\theta(u,\bar x,\bar y)\in[L]$,  $J\in \alpha^{|\bar x|}$, and $\bar{\bo {c}}\in C^{|\bar y|}$,
$$\l \theta(\bo u,A'_J,\bar{\bo c})\rr\doteq\l \theta(\tau_\alpha({\bo a}_\alpha),A'_J,\bar{\bo c})\rr.$$
\item[(b)] For every algebraical formula $\psi(u,\bar y)\in[L]$ and finite $\bar{\bo b}\subseteq A'_\alpha BC$,
$$\l\psi({\bo u},\bar{\bo b})\rr\sqsubseteq\l{\bo a}_\alpha = \bo a_0\rr.$$
\end{itemize}

\emph{Proof of Claim 2.}
We must show that the set of $L^R$-formulas $\Gamma(\bo u)\cup\Delta(\bo u)$ is satisfiable in $\cu N$, where
$$\Gamma(\bo u)=\{\l\theta(\tau_\alpha(\bo a_\alpha),A'_J,\bar{\bo c})\rr\sqsubseteq\l\theta(\bo u,A'_J,\bar{\bo c})\rr\mid
 \theta\in[L],J\subseteq \alpha,\bar{\bo c}\in C^{<\BN}\},$$
$$\Delta(\bo u)=\{\l\psi(\bo u,\bar{\bo b})\rr\sqsubseteq\l \bo a_\alpha=\bo a_0\rr\mid \psi \mbox{ algebraical, } \bar{\bo b}\subseteq A'_\alpha BC\}.$$
The argument is similar to the proof of Claim 1 in Lemma \ref{l-single}, but with a different set of formulas.
We will implicitly make use of the fact that each event we mention belongs to $\cu{Z}$ and is therefore fixed by the automorphism $\tau_\alpha$.
In particular,
$$\l{\bo a}_\alpha=\bo a_0\rr\doteq\tau_\alpha(\l{\bo a}_\alpha=\bo a_0\rr)\doteq \l \tau_\alpha({\bo a}_\alpha)=\bo a'_0\rr$$
and
$$\l\theta(\tau_\alpha({\bo a}_\alpha),A'_J,\bar{\bo c})\rr\doteq
\tau_\alpha(\l\theta({\bo a}_\alpha,A_J,\bar{\bo c})\rr)\doteq\l\theta({\bo a}_\alpha,A_J,\bar{\bo c})\rr.$$
Let $\Gamma_0(\bo u), \Delta_0(\bo u)$ be finite subsets of $\Gamma(\bo u)$ and $\Delta(\bo u)$.
At almost every $\omega\in\l\bo a_\alpha=\bo a_0\rr$, $\Gamma_0(\tau_\alpha(\bo a_\alpha))\cup\Delta_0(\tau_\alpha(\bo a_\alpha)).$ holds.
By hypothesis (2), for each first order formula $\theta(u,\bar x,\bar z)$ and parameters $\bar{\bo c}\in C^{<\BN}$,
at almost all $\omega\in \l\bo a_\alpha\ne\bo a_0\rr,$ if $\theta(\tau_\alpha(\bo a_\alpha),A'_J,\bar{\bo c})(\omega)$ holds in $\cu M$ then
$\theta(u,A'_J,\bar{\bo c})(\omega)$ is satisfied by infinitely many $u$ in $\cu M$.
 Therefore at almost all $\omega\in \l\bo a_\alpha\ne\bo a_0\rr,$ $\Gamma_0(\bo u)$ is satisfied
by infinitely many elements of $\cu M$, and $\Delta_0(\bo u)$ is satisfied by all but finitely many elements of $\cu M$.
Since $\cu N$ has perfect witnesses, it follows that $\Gamma_0(\bo u)\cup\Delta_0(\bo u)$ is satisfiable in $\cu N$.  By the saturation of $\cu N$,
$\Gamma(\bo u)\cup\Delta(\bo u)$ is satisfiable in $\cu N$.  This proves Claim 2.
\medskip

We now take $\bo a'_\alpha$ to be an element $\bo u$ of $\cu K$ that satisfies conditions (a) and (b) of Claim 2.
By (1), every quantifier-free formula of $L^R$
has the same truth values at the parameters $A'_{\alpha+1}C\cu{Z}$ as at the parameters $A_{\alpha+1}C\cu{Z}$. Thus by quantifier elimination in $\cu N$,
$$ A'_{\alpha + 1}\equiv_{C \cu{Z}} A_{\alpha+1}.$$
To complete the induction, it remains to prove that $A'_{\alpha+1}\ind[\th\omega]_C \ B$.
By Lemma  \ref{l-ind[eomega]-easy}, $\ind[\th\omega] \ \ $ has finite character and monotonicity,
so it suffices to show that for every finite $I\subseteq\alpha+1$ such that $0\in I$,
$A'_I\ind[\th\omega]_{C} \ B$.  We will prove this by induction on the cardinality of $I$.  By the choice of $\bo a'_0$, we have
$\bo a'_0\ind[\th\omega]_{C} \ B$.  If $I\subseteq\alpha$, then $A'_I\ind[\th\omega]_{C} \ B$ follows from monotonicity and the hypothesis that (3)
holds when $\beta=\alpha$.  So it is enough to prove $A'_I\ind[\th\omega]_C \ B$ under the assumption that $0\in J\subseteq \alpha$, $I=J\cup\{\alpha\}$,
and  $A'_J\ind[\th\omega]_C \ B$.  Note that because Claim 2 (a) holds when $\bo u=\bo a'_\alpha$,
$$\l\bo a_\alpha=\bo a_0\rr\doteq\l\tau_\alpha(\bo a_\alpha)=\bo a'_0\rr\doteq\l\bo a'_\alpha=\bo a'_0\rr.$$
Let $B_0\subseteq B$ and $C_0\subseteq C$ be countable.  Then there is a countable $C_1\in[C_0,C]$ such that
 $\mu(\l A'_{J}\ind[\th]_{C_1} B_0\rr)=1$.  Now consider any tuple
$\bar{\bo a}'\subseteq A'_J$. Because $0\in J$, we have ${\bo a}'_0\bar{\bo a}'\subseteq A'_J$, so by monotonicity,
$\mu(\l{\bo a}'_0\bar{\bo a}'\ind[\th]_{C_1} B_0\rr)=1.$  Hence
$$ \l {\bo a}'_0\in\acl(\bar{\bo a}' B_0 C_1)\rr\sqsubseteq \l{\bo a}'_0\in\acl(\bar{\bo a}'C_1)\rr.$$
Because Claim 2 (b) holds when $\bo u=\bo a'_\alpha$, we have
$$\l {\bo a}'_\alpha\in\acl(\bar {\bo a}' B_0 C_1)\rr\sqsubseteq\l\bo a'_\alpha=\bo a'_0\rr,$$
and hence
$$\l {\bo a}'_\alpha\in\acl(\bar {\bo a}' B_0 C_1)\rr\sqsubseteq\l {\bo a}'_0\in\acl(\bar {\bo a}' C_1)\rr,$$
so
$$\l {\bo a}'_\alpha\in\acl(\bar {\bo a}' B_0 C_1)\rr\sqsubseteq\l {\bo a}'_\alpha\in\acl(\bar {\bo a}' C_1)\rr.$$
This means that $\mu(\l {\bo a}'_\alpha\thind_{\bar{\bo a}' C_1} \ B_0\rr)=1,$
and by the Pairs Lemma (Result \ref{r-pairs}), $\mu(\l {\bo a}'_\alpha \bar{\bo a}'\thind_{C_1} B_0\rr)=1.$
This shows that $\mu(\l A'_I\thind_{C_1} B_0\rr)=1,$ so we have $A'_I\ind[\th\omega]_C \ B$ as required.

Finally, since $\ind[\th\omega] \ \ \ $ has finite character, it follows that $A'\equiv_{C\cu{Z}} A$ and $A'_{\alpha+1}\ind[\th\omega]_C \ B$.
This proves Lemma \ref{l-wlog}.
\end{proof}

We are now ready to prove Theorem \ref{l-full-existence}, showing that $\ind[\th\omega] \ \ $ has full existence.

\begin{proof}[Proof of Theorem \ref{l-full-existence}]
Let $D=\{{\bo d}_\alpha\mid\alpha<\kappa\}$. We must construct a set $D'\equiv_{C\cu{Z}} D$ such that $D'\ind[\th\omega]_C \ B.$
We may assume without loss of generality that $\cu{Z}\supseteq\acl_\BB(BCD)$.
In order to use Lemma \ref{l-wlog}, we inductively define a new enumerated set $A=\{{\bo a}_\alpha\mid \alpha<\kappa\}$ such that for each $\alpha<\kappa$,
\begin{itemize}
\item[(a)] $A_\alpha\subseteq\acl(CD_\alpha)$,
\item[(b)]
$(\forall^c G\subseteq A_\alpha C)\l \bo a_\alpha \in\acl(G)\rr\sqsubseteq\l \bo a_\alpha = {\bo a}_0\rr.$
\end{itemize}
We first put ${\bo a}_0={\bo d}_0$.
Suppose $\alpha>0$ and we have already defined $A_\alpha=\{{\bo a}_\beta\mid\beta<\alpha\}$, and that $A_\alpha\subseteq\acl(CD_\alpha)$.
By Lemma \ref{l-blanket-local-char} there is a
countable set $X_\alpha\subseteq A_\alpha C$ that blankets $\bo d_\alpha$ in $A_\alpha C$. Let
$$ \bo a_\alpha(\omega) = \begin{cases} \bo a_0(\omega) \mbox{ if } \bo d_\alpha(\omega)\in \acl(X_\alpha(\omega))\\
                                        \bo d_\alpha(\omega) \mbox{ otherwise}.
                        \end{cases}
$$
We have
$$\l {\bo d}_\alpha\in \acl(X_\alpha)\rr\in\sigma(\fo_\BB(CD_{\alpha+1}))=\acl_\BB(CD_{\alpha+1}),$$
so  $\bo a_\alpha\in\acl(CD_{\alpha+1})$ and (a) holds.  Moreover,
$$\l \bo d_\alpha\in\acl(X_\alpha)\rr\sqsubseteq \l \bo a_\alpha = \bo a_0\rr \mbox{ and }\l \bo a_\alpha\ne\bo d_\alpha\rr\sqsubseteq\l\bo a_\alpha = \bo a_0\rr.$$
Since $X_\alpha$ blankets $\bo d_\alpha$ in $A_\alpha C$, for each algebraical formula $\psi(u,A_\alpha C)$ we have
$$\l\psi({\bo d}_\alpha,A_\alpha C)\rr\sqsubseteq\bigcup_{\bar{\bo x}\subseteq X_\alpha}\l\psi({\bo d}_\alpha,\bar{\bo x})\rr.$$
Therefore for each countable set $G\subseteq A_\alpha C$,
$$\l \bo d_\alpha\in \acl(G)\rr\sqsubseteq\l \bo d_\alpha\in\acl(X_\alpha)\rr\sqsubseteq\l\bo a_\alpha=\bo a_0\rr,$$
so
$$(\l \bo a_\alpha =\bo d_\alpha\rr\cap\l \bo a_\alpha\in \acl(G)\rr)\sqsubseteq\l\bo a_\alpha=\bo a_0\rr.$$
 It follows that (b) holds, and the inductive definition is complete.  Thus $A$ satisfies the hypotheses of Lemma \ref{l-wlog}.
By Lemma \ref{l-wlog}, there exists $A'\equiv_{C\cu{Z}} A$ such that $A'\ind[\th\omega]_C \ B.$
By the saturation of $\cu N$, there is an automorphism $\bo u\mapsto\bo u'$ of $\cu N$ that is fixed on $C\cu{Z}$ and sends $\bo a_\alpha$ to $\bo a'_\alpha$
for each $\alpha<\kappa$.  Then $D'\equiv_{C\cu{Z}} D$.   We will show that $D'\ind[\th\omega]_C \ B.$

 Each countable subset of $D'$ is equal to $D'_I$ for some countable $I\subseteq\kappa$.
Fix a countable $I\subseteq\kappa$, and let $ C_0, B_0$ be countable subsets of $ C, B$ respectively.
For each $\beta\in I$, $X_\beta$ is a countable subset of $A_\beta C$, so there are countable $J\subseteq\kappa$ and $C_1\in [C_0,C]$
such that $J\supseteq I$ and $\bigcup_{\beta\in I}X_\beta\subseteq A_J C_1$.

By the definition of ${\bo a}_\beta$, we always have
$${\bo d}_\beta(\omega)\in\acl(X_\beta)(\omega)\cup\{{\bo a}_\beta(\omega)\}.$$
For each $\beta\in I$, $X_\beta\cup\{{\bo a}_\beta\}\subseteq A_J C_1$.  Therefore
$$\mu(\l D_I\subseteq\acl(A_J C_1)\rr)=\mu(\l D'_I\subseteq\acl(A'_J C_1)\rr)=1.$$
  Since $A'\ind[\th\omega]_C \ B$, we have
$\mu(\l A'_J\ind[\th]_{C_2} B_0\rr)=1$ for some countable set $C_2\in[C_1,C]$.  It is obvious from the definition of $\ind[e]$ that $B\ind[e]_C A\Rightarrow B\ind[e]_C \acl(AC)$
in $\cu M$, so by symmetry we have
 $\mu(\l \acl(A'_J C_2)\ind[\th]_{C_2} B_0\rr)=1$. By monotonicity, $\mu(\l D'_I\ind[\th]_{C_2} B_0\rr)=1.$
Therefore $D'\ind[\th\omega]_C \ B$.
\end{proof}

\begin{thm}  \label{t-pointwise-ind[e]}  Suppose $T$ has the exchange property.  Then
$\ind[\th\omega] \ \ $ is a pointwise anti-reflexive independence relation on $\cu N$ with small local character.
\end{thm}

\begin{proof}  By Lemma \ref{l-ind[eomega]-easy} and Theorem \ref{l-full-existence}.
\end{proof}

\begin{cor}  Suppose $T$ is o-minimal. Then in $T^R$,
$\ind[\th\omega] \ \ $ is a pointwise anti-reflexive independence relation on $\cu N$ with small local character.
\end{cor}

\begin{proof}  Every o-minimal theory has the exchange property.
\end{proof}

\begin{cor}
Suppose $T$ has the exchange property and $\ind[a]$ is an independence relation in $T$.  Then in $T^R$, $\ind[a\omega] \ $
is a pointwise anti-reflexive independence relation with small local character.
\end{cor}

\begin{proof}  Since $\ind[a]$ is an independence relation, $\ind[a]=\thind$ in $T$.  The result now follows from Theorem \ref{t-pointwise-ind[e]}.
\end{proof}

\section{Real Rosy Randomizations}  \label{s-rosy}

\subsection{Main Theorem}

We know from Theorem \ref{t-pointwise-ind[e]} above that if $T$ has the exchange property then $\cu N$ has a pointwise anti-reflexive independence relation.  In Theorem \ref{t-main} below we will prove that if $T$ has the exchange property and also has $\acl=\dcl$, then $\cu N$ has an anti-reflexive independence relation, so $T^R$ is real rosy.
In the following, let $\ind[\th \wedge] \ =\ind[\th\omega] \ \wedge\ind[d\BB] \ $.  Recall that if $T$ has the exchange property, then $\thind=\ind[e]$ in $T$,

\begin{thm}  \label{t-main}
Suppose $T$ has the  exchange property and $\acl=\dcl$.
Then in $T^R$, $\ind[\th \wedge] \ $ is a strict independence relation that has finite character, has small local character, and is countably based.
Moreover, $T^R$ is real rosy, and in $T^R$, $\thind$ is a strict countable independence relation that
has small local character and is countably based.
\end{thm}

Before proving \ref{t-main} we state two corollaries whose hypotheses imply that $T$ has $\acl=\dcl$ and the exchange property.

\begin{cor}  If $\acl(A)=A$ for all $A$ in the big model of $T$, then $T^R$ is real rosy.
\end{cor}

\begin{cor} \label{c-main-omin} If $T$ is  o-minimal, then $T^R$ is real rosy.
\end{cor}

\begin{proof}[Proof of Theorem \ref{t-main}]
Using only the hypothesis that $T$ has the exchange property, we show that $\ind[\th \wedge] \ \ $ is an independence relation with small local character
and is countably based.  Then
using the hypothesis that $T$ has $\acl=\dcl$, we show that $\ind[\th \wedge] \ \ $ is also anti-reflexive.

\emph{Basic axioms and finite character:}  By Lemmas \ref{l-fB-basic} and \ref{l-ind[eomega]-easy},
both $\ind[d\BB] \ $ and $\ind[\th\omega] \ \ $
satisfy the basic axioms and finite character.  It follows easily that $\ind[\th \wedge] \ \ $ satisfies the basic axioms and has finite character.

\emph{Small local character:}  Let $A, B$ be small subsets of $\cu K$.  By Lemmas \ref{l-fB-local} and \ref{l-ind[eomega]-easy},
there are sets $C_1, C_2\subseteq B$ of cardinality $\le |A|+\aleph_0$ such that $A\ind[\th\omega]_{C_1} \ B $ and $A\ind[d\BB]_{C_2} \ B$.
By base monotonicity for both $\ind[\th\omega] \ $ and $\ind[d\BB] \ $, we have $A\ind[\th \wedge]_{C_1\cup C_2} \ B$.

\emph{Full existence:}  Suppose $A, B, C$ are small subsets of $\cu K$.  By full existence for $\ind[d\BB] \ \ ,$  there is an automorphism $\tau$
of $\cu N$ that fixes $C$ such that $\tau(A)\ind[d\BB]_C \ B.$   Let $\cu{Z}=\acl_\BB(\tau(A)BC).$  By Theorem \ref{l-full-existence}, there
is an automorphism $\pi$ of $\cu N$ that fixes $C\cu{Z}$ such that $\pi(\tau(A))\ind[\th\omega]_C \ \ B.$
Then $\pi(\sa E)=\sa E$ for every $\sa E\in\cu Z$, so  $\pi(\tau(A))\ind[d\BB]_C \ B.$ Therefore $\pi(\tau(A))\equiv_C A$
and $\pi(\tau(A))\ind[\th \wedge]_C \ B,$ so $\ind[\th \wedge] \ \ $ has full existence.

\emph{Countably based:}  $\ind[\th\omega] \ $ is countably based by definition, and $\ind[d\BB] \ $ is countably based by Lemma \ref{l-fB-cbased}.
By Lemmas \ref{l-ind[e]-union}, \ref{p-omega-cbased}, and \ref{l-fB-basic}, $\ind[\th\omega] \ $ and $\ind[d\BB] \ $ have the countable union property.
Then by Proposition \ref{p-union-pair}, $\ind[\th \wedge] \ \ $ is countably based.

\emph{Anti-reflexive:}  Suppose $\bo a\ind[\th \wedge]_C  \ \bo a$.  By Lemma \ref{l-ind[fB]-implies-ind[aB]}, $\bo a\ind[a\BB]_C \ \ \bo a.$
By Lemma \ref{l-ind[eomega]-easy}, $\ind[\th\omega] \ \ $ is pointwise anti-reflexive, so $\bo a\in\acl^\omega(C)$.   Hence there is a countable $D\subseteq C$ such that
$\bo a(\omega)\in\acl(D(\omega)) \as.$  We are assuming that $T$ has $\acl=\dcl$, so
$\bo a$ is pointwise definable over $D$. For every functional formula $\psi(u,\bar v)$ and tuple $\bar{\bo c}\in C^{<\BN}$,
we have $\l\psi(\bo a,\bar{\bo c})\rr\in \dcl_\BB(\bo a C)$.  But $\bo a\ind[a\BB]_C \  \bo a,$ so
$\acl_\BB(\bo a C)\subseteq \acl_\BB(C)$.  Then by Result \ref{f-acl=dcl}, $\dcl_\BB(\bo a C)\subseteq \dcl_\BB(C)$, and by Result \ref{f-dcl3}, $\bo a\in\dcl(C)\subseteq\acl(C)$.

By Result \ref{f-weakest}, $\ind[\th \wedge] \ \Rightarrow\thind$.  By Remark \ref{r-weaker},
$\thind$ has small local character, so $T^R$ is real rosy.  By Proposition \ref{p-thorn-cbased}, $\thind$ is countably based.
\end{proof}

\subsection{Some Questions}

\begin{question}  Suppose $T$ has the exchange property and $\acl=\dcl$.  In $T^R$, does $\ind[\th \wedge] \ =\thind$?  Does $\thind$  have finite character?
\end{question}


By Theorem \ref{t-indep-FO}, if $T$  has the exchange property then $\ind[M] \ $ for $T$ has symmetry.  By Corollary 5.4 of [Ad3], if $\ind[M] \ $ for $T$ has symmetry,
then $T$ is real rosy.  So it is natural to ask the following question.

\begin{question}  Suppose $\ind[M] \ $ for $T$ has symmetry and $T$ has $\acl=\dcl$.  Is $T^R$ real rosy?
\end{question}

In view of Corollary \ref{c-main-omin}, we ask:

\begin{question} \label{q-rosy}  If $T$ is o-minimal, must $T^R$ be rosy, or at least rosy for some natural family of imaginaries?
\end{question}

It follows from [VDD] that every o-minimal expansion of an ordered abelian group with a positive constant eliminates imaginaries.
In [EG] it is shown that every real rosy first order theory that (weakly) eliminates imaginaries is rosy, and every real rosy continuous theory that (weakly) eliminates finitary imaginaries is rosy for finitary imaginaries.  This leads to the following question.

\begin{question} \label{q-elim} If $T$  eliminates imaginaries, must $T^R$ eliminate finitary imaginaries, or at least eliminate some natural family of imaginaries?
\end{question}

In [EG], many families of imaginaries and corresponding variants of the notion of rosiness are considered in continuous logic.  This suggests that Questions \ref{q-rosy} and \ref{q-elim} may be major projects.
Proposition \ref{r-elim} below gives an easy affirmative answer to Question \ref{q-elim} for a very restricted family of imaginaries.

We say that a formula (or finitary predicate) $\Theta(\bar x,\bar z)$ \emph{eliminates imaginaries} for a formula
(or finitary predicate) $\Phi(\bar x,\bar y)$ if for every $\bar{\bo b}$ there is a unique $\bar{\bo c}$ such that for all $\bar{\bo a}$ we have
$\Phi(\bar{\bo a},\bar{\bo b})=\Theta(\bar{\bo a},\bar{\bo c}).$  Recall from [Po] that a first order theory eliminates imaginaries if and only if for every formula $\varphi$ there is a formula $\theta$ that eliminates imaginaries for $\varphi$.  By [EG], a continuous theory eliminates finitary imaginaries if and only if for every finitary predicate $\Phi$ there is a finitary predicate $\Theta$ that eliminates imaginaries for $\Phi$.

\begin{prop}  \label{r-elim}  If
$\theta(\bar x,\bar z)$ eliminates imaginaries for $\varphi(\bar x,\bar y)$ in $\cu{M}$, then the formula
$\mu(\l\theta(\bar x,\bar z)\rr\sqcap\sa X)$ eliminates imaginaries for $\mu(\l\varphi(\bar x,\bar y)\rr\sqcap\sa X)$ in $\cu N$.
\end{prop}

\begin{proof}  Take a tuple $\bar{\bo b}$ in $N$.  We have
$$\mu(\l (\exists \bar z)(\forall \bar x)[\varphi(\bar x,\bar{\bo b})\leftrightarrow\theta(\bar x,\bar z)]\rr)=1.$$
We show that in $\cu N$ there is a unique tuple $\bar{\bo c}$ such that
\begin{equation}  \label{eq-elim}
(\forall \bar{\bo{a}})(\forall\sa A)
[\mu(\l \varphi(\bar {\bo a},\bar{\bo b})\rr\sqcap\sa A)=\mu(\l\theta(\bar {\bo a},\bar{\bo c})\rr\sqcap\sa A)].
\end{equation}
By Fullness, there is a tuple $\bar{\bo c}$ in $N$ such that
$$\mu(\l (\forall \bar x)[\varphi(\bar x,\bar{\bo b})\leftrightarrow\theta(\bar x,\bar{\bo c})]\rr)=1.$$
Then for every $\bar{\bo a}$ in $N$, we have
$\l \varphi(\bar {\bo a},\bar{\bo b})\rr=\l\theta(\bar {\bo a},\bar{\bo c})\rr,$ so (\ref{eq-elim}) holds for $\bar{\bo c}$.

To prove uniqueness, suppose  $\bar{\bo d}\ne \bar{\bo c}$.  Then $\mu(\l \bar{\bo d}\ne \bar{\bo c}\rr)>0$.  Let
$$\sa E=\l (\exists \bar x)(\varphi(\bar x,\bar{\bo b})\wedge \neg\theta(\bar x,\bar{\bo d}))\rr$$
and
$$\sa F=\l (\exists \bar x)(\neg\varphi(\bar x,\bar{\bo b})\wedge \theta(\bar x,\bar{\bo d}))\rr.$$
Then either $\mu(\sa E)>0$ or $\mu(\sa F)>0$.  Suppose $\mu(\sa E)>0$. By Fullness, there is a tuple $\bar{\bo a}$ in $N$ such that
$$\sa E\subseteq\l\varphi(\bar{\bo a},\bar{\bo b})\wedge \neg\theta(\bar{\bo a},\bar{\bo d})\rr.$$
Then
$$ \mu(\l\theta(\bar{\bo a},\bar{\bo d})\rr\sqcap\sa E)=0<\mu(\sa E)=\mu(\l\varphi(\bar{\bo a},\bar{\bo b})\rr\sqcap\sa E).$$
Therefore
$$(\exists \bar{\bo{a}})(\exists\sa A)
[\mu(\l \varphi(\bar {\bo a},\bar{\bo b})\rr\sqcap\sa A)\ne\mu(\l\theta(\bar {\bo a},\bar{\bo c})\rr\sqcap\sa A)],$$
and (\ref{eq-elim}) fails for $\bar{\bo d}$.  By a similar argument, if $\mu(\sa F)>0$ then (\ref{eq-elim}) fails for $\bar{\bo d}$.
This contradiction shows that $\bar{\bo c}$ is the unique tuple that satisfies (\ref{eq-elim}), as required.
\end{proof}

\section*{Review of independence relations}

 $\ind[a]$ (\emph{Algebraic}):
$A\ind[a]_CB$ iff $\acl(AC)\cap \acl(BC)=\acl(C)$.
\bigskip

 $\ind[M]$ \ :
$A\ind[M]_CB$ iff for every $D\in[C,\acl(BC)]$, we have $A\ind[a]_{D}B$.
\bigskip

 $\thind$ (\emph{Thorn}):
$A\thind_CB$ iff for every (small) $E\supseteq BC$, there is $A'\equiv_{BC}A$ such that $A'\ind[M]_CE$.
\bigskip

$\ind[d] \ $ (\emph{Dividing}):
In continuous logic, $A\ind[d]_C B$ iff there is no tuple $\bar a\in A^{<\BN}$ and  formula
$\Phi(\bar x,B,C)$ such that $\Phi(\bar a,B,C)=0$ and $\Phi(\bar x, B,C)$ divides over $C$.
\bigskip

$\ind[I], \ \ind[J], \ \ldots.$:  \emph{Arbitrary ternary relations}
\bigskip

$\ind[Ic] \ $ : The unique countably based relation that agrees with $\ind[I]$ on countable sets.
\bigskip

$\ind[b] \ $ (\emph{Blanketing}   (first order)):
$A\ind[b]_C B \ $ iff for every
$\varphi(\bar x,\bar y,\bar z)\in[L]$ and all tuples $\bar{ a}\in A^{|\bar x|}$, $\bar{ b}\in B^{|\bar y|}$ and  $\bar{ c}\in C^{|\bar z|}$,
there exists $\bar{ d}\in C^{|\bar y|}$ such that
$$ \cu M\models\varphi(\bar{ a},\bar{ b},\bar{ c})\Rightarrow\varphi(\bar{ a},\bar{ d},\bar{ c}).$$

$\ind[e] \ $ (\emph{Exchange})  (first order):
$A\ind[e]_C B$ iff for all finite $\bar a\in A^{<\BN}$,
$$ A\cap \acl(\bar a B C)\subseteq \acl(\bar a C).$$

$\ind[d\BB] \ $ (\emph{Dividing in event sort}): $ A\ind[d\BB]_C \ B \Leftrightarrow \cu A_C\ind[d]_{\cu C} \cu B_C \mbox{ in } (\cu E, \mu).$
\bigskip

$\ind[a\BB] \ $ (\emph{Algebraic in event sort}): $A\ind[a\BB]_C \ \ B\Leftrightarrow\acl_\BB(AC)\cap\acl_\BB(BC)=\acl_\BB(C).$
\bigskip

$\ind[I\omega] \ $ (\emph{Pointwise}): Countably based with $ A\ind[I \omega]_C \  B\Leftrightarrow A(\omega)\ind[I]_{C(\omega)} B(\omega) \ \as.$
\bigskip

\emph{Pointwise and event dividing}:  $\ind[d \wedge] \ \ =\ind[d\omega] \ \wedge \ind[d\BB] \ $.
\bigskip

\emph{Pointwise and event thorn}:  $\ind[\th \wedge] \ =\ind[\th\omega] \ \wedge\ind[d\BB] \ \ $.

\section*{References}


\vspace{2mm}

[Ad1]  Hans Adler. Explanation of Independence.  PH. D. Thesis, Freiburg, AxXiv:0511616 (2005).

[Ad2]  Hans Adler.  A Geometric Introduction to Forking and Thorn-forking.  J. Math. Logic 9 (2009), 1-21.

[Ad3]  Hans Adler.  Thorn Forking as Local Forking.  Journal of Mathematical Logic 9 (2009), 21-38.

[AGK]  Uri Andrews, Isaac Goldbring, and H. Jerome Keisler.  Definable Closure in Randomizations.
Annals of Pure and Applied Logic 166 (2015), pp. 325-341.

[AK]  Uri Andrews and H. Jerome Keisler.  Separable Randomizations of Models.  Journal of Symbolic Logic 80 (2015), 1149-1181.

[Be1]  Ita\"i{} Ben Yaacov. Schrodinger's Cat.  Israel J. Math. 153 (2006),  157-191.

[Be2]  Ita\"i{} Ben Yaacov.  On Theories of Random Variables.  Israel J. Math 194 (2013), 957-1012.

[Be3]  Ita\"i{} Ben Yaacov.  Simplicity in Compact Abstract Theories.  Journal of Mathematical Logic 3 (2003), 163-191.

[BBHU]  Ita\"i{} Ben Yaacov, Alexander Berenstein,
C. Ward Henson and Alexander Usvyatsov. Model Theory for Metric Structures.  In Model Theory with Applications to Algebra and Analysis, vol. 2,
London Math. Society Lecture Note Series, vol. 350 (2008), 315-427.

[BK] Ita\"i{} Ben Yaacov and H. Jerome Keisler. Randomizations of Models as Metric Structures.
Confluentes Mathematici 1 (2009), pp. 197-223.

[BU] Ita\"i{} Ben Yaacov and Alexander Usvyatsov. Continuous first order logic and local stability. Transactions of the American
Mathematical Society 362 (2010), no. 10, 5213-5259.

[CK]  C.C.Chang and H. Jerome Keisler.  Model Theory.  Dover 2012.

[EG]  Clifton Ealy and Isaac Goldbring.  Thorn-Forking in Continuous Logic.  Journal of Symbolic Logic
77 (2012), 63-93.

[GL1]  Isaac Goldbring and Vinicius Lopes.  Pseudofinite and Pseudocompact Metric Structures.  Notre Dame Journal of Formal Logic, Volume 56 (2015), 493-510.

[GL2]  Rami Grossberg and Olivier Lessman.  Dependence Relation in Pregeometries. Algebra Universalis 44 (2000), 199-216.

[HP] Ehud Hrushovski and Anand Pillay.  Groups definable in local fields and pseudo-finite fields.  Israel J. Math 85 (1994), 203--262.

[Ke1] H. Jerome Keisler.  Randomizing a Model.  Advances in Math 143 (1999), 124-158.

[Ke2] H. Jerome Keisler.  Finite Approximations of Infinitely Long Formulas, pp. 158-169 in Theory of Models, edited by J.Addison, L. Henkin, and A. Tarski,
North-Holland, Amsterdam, 1965.

[On]  Alf Onshuus, Properties and Consequences of Thorn Independence. J. Symbolic Logic 71 (2006), 1-21.

[Po]  Bruno Poizat. A Course in Model Theory: An Introduction to Contemporary Mathematical Logic.  Universitext, Springer-Verlag 2000.

[VDD] Lau van den Dries.  Tame Topology and O-minimal Structures.  Cambridge University  Press, 1998.

\end{document}